\let\ORIlabel\label
\let\ORIrefstepcounter\refstepcounter
\AddToHook{package/hyperref/before}{\let\label\ORIlabel\let\refstepcounter\ORIrefstepcounter}

\PassOptionsToPackage{dvipsnames}{xcolor}
\documentclass[leqno]{siamart1116}
\usepackage{amssymb}
\usepackage{bm,bbm}
\usepackage{derivative}
\usepackage{dsfont}
\usepackage{dutchcal}
\usepackage{enumerate}
\usepackage{enumitem}
\usepackage{graphicx}
\usepackage{mathtools}
\usepackage{multirow}
\usepackage[numbers, sort]{natbib}
\usepackage{subcaption}
\usepackage{booktabs}
\usepackage{xspace}
\usepackage{pgfplots}

\pgfplotsset{plotstyle/.style={
width=0.6\columnwidth,
height=0.4\columnwidth,
scale only axis,
xminorticks=true,
xlabel style={font=\normalsize},
yminorticks=true,
ylabel style={font=\normalsize},
axis background/.style={fill=white},
xmajorgrids,
xminorgrids,
ymajorgrids,
yminorgrids,
legend style={at={(0.5,1.03)}, anchor=south, legend columns=4, legend cell align=left, align=left, draw=white!15!black}
}
}

\newenvironment{customlegend}[1][]{%
        \begingroup
        \csname pgfplots@init@cleared@structures\endcsname
        \pgfplotsset{#1}%
    }{%
        \csname pgfplots@createlegend\endcsname
        \endgroup
    }%

\usepackage[algo2e,linesnumbered]{algorithm2e}
\usepackage{algorithmic}

\usepackage{siunitx}
\DeclareMathOperator*{\conv}{conv}

\usepackage{mydef}

\begin{document}

\newcommand\footnotemarkfromtitle[1]{%
    \renewcommand{\thefootnote}{\fnsymbol{footnote}}%
    \footnotemark[#1]%
    \renewcommand{\thefootnote}{\arabic{footnote}}}

\newcommand{\TheTitle}{%
  A conservative invariant-domain preserving projection technique for hyperbolic systems under adaptive mesh refinement}
\newcommand{\TheAuthors}{J.~J. Harmon, M. Kronbichler, M. Maier, E.~J. Tovar}

\headers{Invariant-domain preserving adaptive mesh refinement}{\TheAuthors}

\title{{\TheTitle}\thanks{Draft version, \today \funding{
  The work of ET and JH was supported by the Laboratory Directed Research
  and Development (LDRD) program at Los Alamos National Laboratory under
  project numbers 20248233CT-IST and 20251079ER. Research conducted at Los
  Alamos National Laboratory is done under the auspices of the National
  Nuclear Security Administration of the U.S. Department of Energy under
  Contract No. 89233218CNA000001. MK was partially supported by the German
  Federal Ministry of Research, Technology and Space through the project
  ``PDExa: Optimized software methods for solving partial differential
  equations on exascale supercomputers'', grant agreement no. 16ME0637K,
  and the European Union – NextGenerationEU. MM has been partially
  supported by the National Science Foundation under grant DMS-2045636, by
  the Air Force Office of Scientific Research, USAF, under grant/contract
  number FA9550-23-1-0007, and by a TAMUS National Laboratory Office (NLO)
  Research Seed Funding (NLO-RSF). The Los Alamos unlimited release number is LA-UR-25-27248.%
}}}

\author{
    Jake~J.~Harmon\footnotemark[2]
    \and
    Martin~Kronbichler\footnotemark[3]
    \and
    Matthias~Maier\footnotemark[4]
    \and
    Eric~J.~Tovar\footnotemark[1]
}

\maketitle

\renewcommand{\thefootnote}{\fnsymbol{footnote}}
\footnotetext[1]{%
  Theoretical Division, Los Alamos National Laboratory, P.O. Box 1663, Los
  Alamos, NM, 87545, USA.}
\footnotetext[2]{%
X Computational Physics Division, Los Alamos National Laboratory, P.O. Box 1663, Los
Alamos, NM, 87545, USA.}
\footnotetext[3]{%
  Faculty of Mathematics, Ruhr University Bochum, Universitätsstr.~150,
  44801 Bochum, Germany.}
\footnotetext[4]{%
  Department of Mathematics, Texas A\&M University 3368 TAMU, College
  Station, TX 77843, USA.}
\renewcommand{\thefootnote}{\arabic{footnote}}

\begin{abstract}
  We propose a rigorous, conservative invariant-domain preserving (IDP) projection technique for hierarchical discretizations that enforces membership in physics-implied convex sets when mapping between solution spaces. When coupled with suitable refinement indicators, the proposed scheme enables a provably IDP adaptive numerical method for hyperbolic systems where preservation of physical properties is essential. In addition to proofs of these characteristics, we supply a detailed construction of the method in the context of a high-performance finite element code. To illustrate our proposed scheme, we study a suite of computationally challenging benchmark problems, demonstrating enhanced accuracy and efficiency properties while entirely avoiding \emph{ad hoc} corrections to preserve physical invariants.
\end{abstract}

\begin{keywords}
  Adaptive mesh refinement, hyperbolic systems, invariant-domain
  preservation, higher-order accuracy, convex limiting
\end{keywords}

\begin{AMS}
  65M60, 65M12, 35L50, 35L65, 76M10
\end{AMS}


\section{Introduction}
\label{sec:introduction}
Hyperbolic partial differential equations (PDEs), and more generally,
problems of hyperbolic \emph{character} govern the models of sophisticated
physical phenomena for many applications of interest -- ranging from the
study of tidal to astronomical flows, with predictive relevance over orders
of magnitude differences in temporal and spatial scales, cf. \cite{berger2021,
guermond2025high, stone2020, berger2024}. In addition to possessing
essential properties, namely the conservation or positivity of relevant physical quantities, hyperbolic systems yield solutions where, even from
globally smooth initial conditions, discontinuities can emerge in finite
time. An adequate numerical method for these problems, cognizant of the
underlying physics, must maintain conservation, among other
characteristics, with certifiable robustness even with
poor solution regularity.

Aside from the analytical challenges inherent to the study of
this class of PDEs, practical challenges emerge in abundance: the choice of discretization ansatz,
time integration, and ensuring robustness have dominated decades of computational
physics research \cite{toro2013}. The demands, in particular, for
full multi-physics, predictive three-dimensional models have driven research and
development investment for adaptive or \textit{active} discretization methods -- 
powerful schemes intended to deliver both accuracy \emph{and} efficiency.

Discretizations without mesh adaptivity -- such as static global mesh refinement or
cases with more
sophisticated manual or \emph{a priori} grid creation -- can often
lead to considerably higher numerical costs compared to
schemes where the mesh dynamically adapts with the evolution of the solution.
As a consequence, mesh adaptivity can extend the practical applicability
of PDE solvers, allowing to resolve physics of interest more efficiently
and attain expected convergence rates.
For adaptive schemes applied to certain
classes of problems (particularly elliptic or parabolic), even exponential
convergence may be obtained in the presence of poor solution regularity
\cite{babuska1981, gui1986I, gui1986II, gui1986III}. For the study of
hyperbolic PDEs, though absent theoretical guarantees of \emph{exponential}
convergence, the adaptive framework of Berger and Oliger continues to
dominate practice \cite{berger1982adaptive, berger1984, berger1984b,
berger2021}.

\renewcommand{\boxed}[1]{{#1}} 
Adaptive schemes typically use the following workflow:
\begin{align*}
  \boxed{\mathrm{SOLVE}}\;\rightarrow\;\boxed{\mathrm{ESTIMATE}}
  \;\rightarrow\;\boxed{\mathrm{MARK}}\;\rightarrow\;
  \boxed{\mathrm{REFINE}}\diagup\boxed{\mathrm{COARSEN}},
\end{align*}
where $\boxed{\mathrm{SOLVE}}$ represents a complete or incomplete update
step, $\boxed{\mathrm{ESTIMATE}}$ constitutes an error estimation (or
indication) procedure, $\boxed{\mathrm{MARK}}$ denotes the attribution of
particular discretization directives based the results of
$\boxed{\mathrm{ESTIMATE}}$, and finally,
$\boxed{\mathrm{REFINE}}\diagup\boxed{\mathrm{COARSEN}}$ signifies the
execution of those directives. Each stage consists of a suite of machinery,
especially when considering sophisticated refinement modalities, see
\cite{bangerth2003adaptive, bangerth2011}. Additional constraints on each component in the
workflow offer vastly different behavior of the overall scheme, with
certain configurations and target problems leading even to proofs of (quasi-)optimality \cite{carstensen2014, feischl2016}. The terminal step, however, masks in its
conceptual simplicity a wealth of considerations for the prolongation and
restriction of an approximate solution to a new solution space.

Now, in an ansatz agnostic context, maintenance of conservation typically relies on a
projection/interpolation procedure~\cite{farrell2009conservative} across mesh hierarchies. Continued adherence of
the approximate solution to an invariant-domain, however, requires
additional mechanisms. 
Such invariant-domains are intricately linked to the
underlying physics and closures, where an associated scheme, in general, is
considered invariant-domain \emph{preserving} (IDP) if for an initial
condition belonging to a convex set, the action of the scheme preserves
membership in that convex set \cite{guermond2016invariant}. 
In the context of adaptivity, recent work on preserving similar properties was carried out in~\cite{kuzmin2010failsafe,bittl2013adaptive,bonilla2020monotonicity}. 
We note here that the novel projection strategy in this work bears some resemblance
to the approaches discussed in~\cite{kuzmin2010failsafe,bittl2013adaptive}.
However, the novelty of the approach in this work is the realization of
enforcing conservation and the IDP property locally as part of a cell-wise
mass projection (see:\S\ref{sec:mass_projection}). Furthermore, the handling of hanging node constraints in
this context is also novel (see:\S\ref{sec:mass-redistribution}).
Higher-order
IDP schemes have, in recent years, delivered significant advantages for the
study of hyperbolic systems for a variety of models \cite{ern2024, guermond2018second, guermond2019invariant, guermond2025high}. An imposition of, e.g.,
positivity, via \emph{ad hoc} procedures falls outside the species of
scheme we target in this paper; the scheme itself must not rely on \emph{a
posteriori}, inconsistent truncation or tuning to maintain physical
invariants. As the target of this manuscript, we focus on the final stage
of the adaptivity workflow, $\boxed{\mathrm{{REFINE}/{COARSEN}}}$, and
enforcement of a conservative, IDP execution such that the entire
workflow---assuming a likewise conservative and IDP
$\boxed{\mathrm{SOLVE}}$ step---is \emph{guaranteed} robust.

To augment IDP capabilities with the full potency of adaptivity, we extend
the results of \cite{maier21b} to a certifiably robust, IDP adaptive
scheme. Without compromising in discretization tailoring, our scheme delivers
provably IDP locally enriched-reduced discretizations. Coupled with
suitable error indicators (for example, of the conserved quantities or
local entropy production), our proposed method renovates long-standing AMR
workflows to guarantee \emph{a priori} the essential characteristics of the
resulting (weak) solutions, delivering a method with
significantly enhanced fidelity and scalability.

\subsection{Contributions}
\begin{enumerate}[label=(\roman*)]
  \item We propose a novel conservative and invariant-domain preserving
    projection technique for hyperbolic conservation laws.
  \item We prove the IDP and conservative properties of the proposed
    scheme; see Lemmas \ref{lem:mass_projection},
    \ref{lem:mass_redistribution}, \ref{lem:conservative}.
  \item We provide a description of the numerical method in the context of
    a high-performance implementation as part of the open-source project
    \texttt{ryujin}~\cite{maier21b}.
  \item We confirm the theoretical properties and demonstrate the
    performance of our method on a suite of challenging benchmark problems spanning the shallow water and single- and multi-species Euler equations.
\end{enumerate}

\subsection{Organization} The remainder of the paper is organized as
follows. In Section \ref{sec:model}, we outline the generic hyperbolic
system of conservation laws. In Section
\ref{sec:preliminaries}, we detail the properties of the underlying finite
element ansatz, from which the weak solution spaces are generated.
Following these preliminaries, we derive the invariant domain preserving
projection, suited both for enrichment and reduction of the solution
spaces, in Section \ref{sec:mass_projection}, with supplementary
implementation details and the refinement criteria discussed in Section
\ref{sec:details}. A series of representative benchmarks is shown in
Section \ref{sec:illustrations} for specific hyperbolic systems such as shallow water equations and the compressible Euler equations, and their
associated invariant domains that we target. The results of these studies confirm the
numerical analysis of Section \ref{sec:mass_projection}, demonstrating the
proposed scheme delivers enhanced accuracy with provably IDP approximate
solutions.

\section{The model problem}
\label{sec:model}
This paper is concerned with numerical solutions under adaptive mesh
refinement of conservative hyperbolic systems in the form:
\begin{equation}\label{eq:model}
  \begin{cases}
    \partial_t \bu + \DIV\polf(\bu) = \bm{0}, &\bx\in\Omega,\quad t>0,\\
    \bu(\bx, 0) = \bu_0(\bx),  &\bx\in\Omega,
  \end{cases}
\end{equation}
where $\bu(\bx, t)\in\polR^m$ is the conserved variable,
$\polf(\bu)\in\polR^{m \times d}$ is the flux, $\Omega\subset\polR^d$ is
the spatial domain of dimension $d$, and $t$ is the temporal coordinate.
For the sake of simplicity, we assume that no external sources are present
in the hyperbolic system. We further assume that the solution $\bu$ belongs
to some convex admissible set $\calA\subset\polR^m$. More specifically, let
$\{\Psi^l(\bu)>0\}_{l=1}^{k}$ denote a set of $k$ quasi-concave constraints
that depend on the hyperbolic system of interest and represent
physical/thermodynamic properties or invariants. We define the generic
admissible set $\calA$ as follows:
\begin{equation}\label{eq:admissible_set}
  \calA\eqq\{\bu\in\polR^m : \Psi^l(\bu)>0,\,l=1,\dots,k\}.
\end{equation}
We recall that with null fluxes on the domain boundary $\partial \Omega$,
the following property holds: $\partial_t \int_\Omega \bu \diff x =
\bm{0}$; that is, mass is conserved over time. Let $\bu^n_h$ be a discrete
solution to~\eqref{eq:model} and let $T_h : \bu_h\upn \mapsto \bu_h\upnp$
be a generic approximation process involving adaptive mesh refinement. The
goal of this work is to construct $T_h$ such that at the discrete level we
have mass conservation and the discrete solution is invariant to the
admissible set (\ie invariant-domain preserving). We consider specific
hyperbolic systems, and their respective admissible sets, in the numerical illustrations section
\S\ref{sec:illustrations}.

\subsection{Invariant-domain preserving approximation}
The proposed projection detailed in \S\ref{sec:mass_projection} relies on
the existence of a non adaptive invariant-domain preserving approximation
for hyperbolic systems. In this work, we adopt the methodology described
in~\citet{guermond2019invariant} and refer to the approach as the
\textit{convex limiting} technique. Generally speaking, the convex limiting
technique comprises three main ingredients: \textup{(i)} a
mass-conservative low-order invariant domain preserving approximation
process; \textup{(ii)} a
mass-conservative high-order approximation that is potentially
invariant-domain violating; a limiting procedure such that the final solution is in the invariant set. The 
limiting procedure relies on local bounds which are often motivated by the
quasi-concave constraints that make up the admissible
set~\eqref{eq:admissible_set}. These three key ingredients for the convex
limiting technique lay the foundation for the projection technique
\S\ref{sec:mass_projection}. 


\section{Preliminaries}
\label{sec:preliminaries}
In this section, we describe the underlying spatial finite element
approximation details.

\subsection{Spatial discretization}
\label{subsec:spatial_discretization}

\begin{figure}
  \centering
  \includegraphics[width=0.5\linewidth]{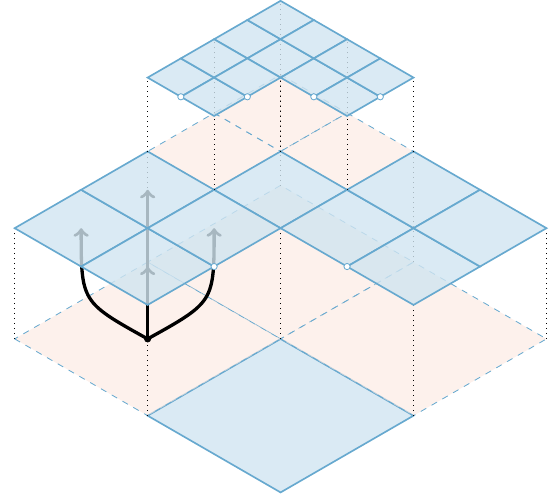}\\
  \includegraphics[width=0.65\linewidth]{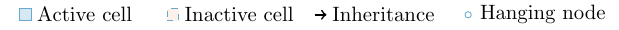}
  \caption{Illustration of hanging nodes and inheritance for a hierarchical discretization.}
  \label{fig:hanging_nodes}
\end{figure}
We start by summarizing the spatial finite element discretization.
We exemplify the derivation using concepts available in the widely used \texttt{deal.II} finite element
library~\citep{dealIIcanonical, dealII96}, albeit the underlying concepts are generic and applicable to many available finite element packages. Let
$\Omega\subset{\mathbb{R}^d}$, $d=1,2,3$, be a computational domain with an
associated partition $\Pi_h$ of mapped, shape-regular cells $K$ that are
obtained from the (quadrilateral or hexahedral) \emph{unit cell} $\hat
K=[0,1]^d$ by a diffeomorphic mapping $T_K:\hat K\to K$. We assume that the
compatibility requirement $T_{K_1}|_{K_1\cap K_2}\equiv T_{K_2}|_{K_1\cap
K_2}$ holds true between any two cells $K_1,K_2\in\Pi_h(\Omega)$. We relax
the usual partition requirement that neighboring cells have to share entire
\emph{faces} and permit 1-irregularity, or one level of hanging nodes, \ie vertices of cells
can also be located on the mapped center of a \emph{face} or \emph{line};
see Figure~\ref{fig:hanging_nodes} \citep{dealIIcanonical, dealII96}. In
practice, such a partition emerges from a hierarchy of
\emph{locally refined} meshes where the coarsest level is a fully regular
mesh without hanging nodes, and in each adaptation cycle only a subset of
cells is marked for coarsening or refinement. Here, a cell is refined by
simple algebraic bisection of the reference element. Furthermore, we constrain coarsening to yield only configurations possible in a previous cycle; namely, coarsening should not disrupt the inheritance tree of the hierarchical mesh.

Let $Q_p(\hat K)$ be the tensor-product space of polynomials
which are of degree less or equal $p$ in each variable $x_1$, \ldots $x_d$.
We now introduce a fully discontinuous and a continuous finite element
space subordinate to $\Pi_h$ as follows:
\begin{align}
  \label{eq:space_discontinuous}
  \mathcal{Q}_p^{\sharp}(\Pi_h) &\eqq \Big\{\varphi\in L^2(\Omega)\;:\;
  \varphi\big|_K\circ T_K(\hat{\vec x})\in Q_p(\hat K) \;\text{for each}\;
  K\in\Pi_h \Big\},
  \\[0.5em]
  \label{eq:space_continuous}
  \mathcal{Q}_{p}(\Pi_h) &\eqq \mathcal{Q}_p^{\sharp}(\Pi_h) \cap
  C^0(\Omega).
\end{align}
Here, $L^2(\Omega)$ is the usual Lebesgue space of measurable and square
integrable functions and $C^0(\Omega)$ is the space of continuous
functions. While mathematically appealing, definition
\eqref{eq:space_continuous} has the problem that it does not immediately
describe an algorithmic realization of the space or of its basis functions.
\texttt{deal.II} addresses this problem by first constructing an
intermediate space that is straightforward to construct by simply
identifying degrees of freedom of neighboring cells that are collocated on
shared geometric subobjects (such as vertices or faces). More precisely,
for any two mapped cells $K_1$, $K_2$ let $S(K_1,K_2)$ denote the union of
\emph{entire} shared geometric subelements (vertices, lines, quads). We
then set:
\begin{multline}
  \label{eq:space_sadreality}
  \tilde{\mathcal{Q}}_{p}(\Pi_h) \eqq \Big\{\varphi \in \mathcal{Q}_p^{\sharp}(\Pi_h)
  \;:\;
  \varphi\big|_{K_1}(\vec x)= \varphi\big|_{K_2}(\vec x)
  \\
  \;\text{for each}\; \vec x\in S(K_1, K_2)
  \text{ for each}\; K_1,K_2\in\Pi_h\Big\},
\end{multline}
Note that for locally refined meshes with hanging nodes, the space in this
intermediate space is not continuous globally. Taking the examples of two adjacent
cells $K_1, K_2$ of
different refinement level, the set $S(K_1, K_2)$ does not include the faces between
the two cells. In \texttt{deal.II} this space \eqref{eq:space_sadreality} is constructed
by iterating over all cells $K\in\Pi_h$, constructing a \emph{local}
enumeration of the Lagrange basis at Gauß--Lobatto points of $Q_p(K)$ on
$K$ and identifying this local enumeration with a global enumeration by
ensuring that collocated degrees of freedom on geometric subobjects are
only counted once. Let $\tilde\calV$ denote a global enumeration of
degrees of freedom of $\tilde{\mathcal{Q}}_{p}(\Pi_h)$ and let
$\big\{\tilde\varphi_i^h(\vec x)\,:\,i\in\tilde\calV\big\}$ be the
corresponding Lagrange basis.

In a second step, a set of algebraic \emph{hanging node} constraints
$\polA\polC_p(\Pi_h)$ are constructed that recover the continuous space
$\mathcal{Q}_{p}(\Pi_h)$ when enforced on the intermediate space
$\tilde{\mathcal{Q}}_{p}(\Pi_h)$. To this end we make the following definition:
\begin{definition}[Algebraic constraints]\label{def:algebraic-constraints}
  \begin{itemize}
    \item
      A \emph{constraint} consists of an index $i\in\tilde{\mathcal{V}}$,
      the \emph{constrained} degree of freedom, and a set of real-valued
      coefficients $\big\{ c^i_j\,:\,j\in\tilde{\mathcal{V}}\big\}$. It
      must hold that $c^i_i=0$.
    \item
      Let $\polA\polC$ be a finite set of algebraic constraints. We require
      that for a given index $i\in\tilde{\mathcal{V}}$ there is at most one
      constraint with that index in $\polA\polC$.
    \item
      We say that $\polA\polC$ is \emph{closed} if for every constraint
      $\big(i,\big\{c_j^i\,:\, j\in\tilde{\mathcal{V}}\big\}\big)$ it holds
      that all $j$ with $c_j^i\not=0$ are unconstrained, meaning
      $\big(j,\ast\big)\not\in\polA\polC$.
    \item
      Let $\polA\polC$ be a closed set of algebraic constraints.
      Furthermore, let $u_h(\bx)=\sum_{i\in\tilde\calV}\alpha_i\tilde\varphi_i^h(\vec
      x)\in\tilde{\mathcal{Q}}_{p}(\Pi_h)$ be arbitrary. We say that the
      algebraic constraints $\polA\polC$ are enforced on $u_h$ if
      for every constraint $\big(i,\big\{c_j^i\,:\,
      j\in\tilde{\mathcal{V}}\big\}\big)$
      holds that $\alpha_i\;=\;\sum_j c^i_j\alpha_{j}$.
  \end{itemize}
\end{definition}
\emph{Hanging node constraints} $\polA\polC_p(\Pi_h)$ are now defined as a
closed set of constraints such that for every $u_h(\vec
x)\in\tilde{\mathcal{Q}}_{p}(\Pi_h)$ the following conditions are
equivalent:
\begin{align*}
  \polA\polC_p(\Pi_h)\text{ are enforced on }u_h
  \quad\Longleftrightarrow\quad
  u_h\in \mathcal{Q}_{p}(\Pi_h).
\end{align*}
In the following we make use of the important observation that for any
entry of the hanging node constraints $\polA\polC_p(\Pi_h)$ it holds true
that
\begin{align*}
  \sum_{j\in\tilde{\mathcal{V}}}c_j^i\,=\,1.
  \color{black}
\end{align*}
For an index $j\in\tilde{\mathcal{V}}$ corresponding to an unconstrained
degree of freedom, we introduce a local \emph{stencil} of constrained
degrees of freedom as follows:
\begin{align*}
  I^{\polA\polC}(j)\eqq\big\{i\in\tilde\calV\,:\,c_{j}^i\not=0\big\}.
\end{align*}

\subsection{Mass matrices and notation}
Recall that $\big\{\tilde\varphi_i(\vec x)\,:\,i\in\tilde\calV\big\}$
denotes the Lagrange basis of $\tilde{\mathcal{Q}}_{p}(\Pi_h)$. We now
introduce local and global mass matrices as follows. For every pair
$i,j\in\tilde\calV$ and every $K\in\Pi_h$ we set:
\begin{align}
  \label{eq:mass_matrices}
  m_{ij}^K \eqq \int_K\tilde\varphi_i^h(\vec x)\tilde\varphi_j^h(\vec
  x)\text{d}x,
  \quad
  m_{i}^K \eqq \int_K\tilde\varphi_i^h(\vec x) \text{d}x,
  \quad
  \tilde m_{i} \eqq \int_\Omega\tilde\varphi_i^h(\vec x) \text{d}x.
\end{align}
We note that the choice of local Lagrange basis functions in Gauß-Lobatto
points ensures that $m_i^K>0$ and $\tilde m_i>0$ for all $i\in\tilde\calV$.
Futhermore, there holds a partition of unity property:
\begin{align*}
  \sum_{i\in\tilde\calV}\sum_{K\in\Pi_h}m_i^K \,=\,
  \sum_{i\in\tilde\calV} \tilde m_i \,=\, |\Omega|.
\end{align*}
As a last step, we introduce a local \emph{stencil} of coupling
indices as follows:
\begin{align*}
  I^K(i)\eqq\big\{j\in\tilde\calV\,:\,m_{ij}^K\not=0\big\}.
\end{align*}


\section{An invariant-domain preserving mass projection}
\label{sec:mass_projection}
We now describe a projection operation that is mass conservative and
preserves an invariant domain. We choose to introduce the projection
operation here in full mathematical generality as a projection from a
general function space into the finite element space,
$P_h:L^\infty(\Omega)^m\to \mathcal{Q}_{p}(\Pi_h)^m$ and prove that the
projection is conservative and invariant domain preserving. We describe the
actual implementation of the projection with the algorithmic realization in
Section~\ref{sec:details}.

Recalling the model assumptions in Section~\ref{eq:model}, we consider an
$m$-dimensional system of hyperbolic conservation equations with a convex
admissible set $\calA\subset\polR^m$,
\begin{align*}
  \calA(\Omega)\;\eqq\;\big\{\bu\in L^\infty(\Omega)^m\;:\;
  \bu(\vec x) \in\calA\;\text{for a.\,e. }\vec x\in\Omega\big\}.
\end{align*}
Here, $L^\infty(\Omega)^m\eqq \underbrace{L^\infty\times \cdots\times
L^\infty}_{m}$. Our goal for this section is to construct a projection
$P_h\;:\;\calA(\Omega)\rightarrow\mathcal{Q}_{p}(\Pi_h)^m$ that is mass
conservative,
\begin{align*}
  \int_\Omega P_h(\bu)\text{d}x\,=\,\int_\Omega \bu\,\text{d}x,
\end{align*}
and preserves an invariant domain. For the latter property we will show
that the projection operation results in projected function values that lie
within a \emph{local} convex hull. That is to say, with
$P_h(\bu)=\sum_i \bU_i\varphi_i^h(\bx)$ it must hold that
\begin{align*}
  \bU_i \;\in\; \conv_{j\in I^{\polA\polC}(i)\cup\{i\}}
  \conv_{\substack{K\in\Pi_h \\ \text{supp}(\tilde\varphi_j^h)\cap K\not=\emptyset}}
  \conv_{\bx\in K}\big(\bu(\bx)\big)\subset\mathcal{A}.
\end{align*}
The convex hull in this condition is constructed over all values
of $\bu(\bx)$ that can possibly influence $\bU_i$ with respect to the
algorithm outlined below. The projection $P_h$ will be composed of an
element-wise mass projection $P_h^{\sharp}$, a nodal averaging projection
$P^{\text{avg}}_h$, and a mass redistribution $P_h^{\text{red.}}$:
\begin{multline*}
  \calA(\Omega)
  \;\xrightarrow{\quad P_h^{\sharp}\quad}\;
  \mathcal{Q}_p^{\sharp}(\Pi_h)^m
  \cap \calA(\Omega)
  \;\xrightarrow{\quad P_h^{\text{avg.}}\quad}\;
  \tilde{\mathcal{Q}}_{p}(\Pi_h)^m
  \cap \calA(\Omega)
  \\[0.5em]
  \;\xrightarrow{\quad P_h^{\text{red.}}\quad}\;
  \mathcal{Q}_{p}(\Pi_h)^m
  \cap \calA(\Omega).
\end{multline*}
We proceed as follows.

\subsection{Element-wise mass projection}\label{sec:element-wise}
We now define the element-wise mass projection: $P_h^{\sharp}(\bu) =
\sum_{i,K} \bU^{\sharp}_{i,K}\tilde \varphi_i^h\big|_K$. Given an
admissible state $\bu\in\calA(\Omega)$, we first define the following
low-order operator and high-order mass projection on each element
$K\in\Pi_h$:
\begin{align*}
    &P_h^{\sharp, L}\;:\;
    \calA(\Omega) \rightarrow \mathcal{Q}_p^{\sharp}(\Pi_h)^m,
    \quad
    P_h^{\sharp,L}(\bu) = \sum_{i,K} \bU^{\sharp,L}_{i,K}\tilde
    \varphi_i^h\big|_K, \\
     &P_h^{\sharp, H}\;:\;
    \calA(\Omega) \rightarrow \mathcal{Q}_p^{\sharp}(\Pi_h)^m,
    \quad
    P_h^{\sharp,L}(\bu) = \sum_{i,K} \bU^{\sharp,H}_{i,K}\tilde
    \varphi_i^h\big|_K,
  \end{align*}
where the states $\bU^{\sharp,L}_i$,  $\bU^{\sharp,H}_i\in\polR^m$, are
chosen such that
\begin{align}
  \label{eq:element_wise_ri}
  \sum_{j}m_{ij}^K \bU^{\sharp,H}_{j,K} = m_i^K \bU^{\sharp,L}_{i,K} = \bR_i^K,
  \quad
  \bR_i^K\eqq \int_K\bu(\bx)\tilde\varphi_i^h\text{d}x,
  \quad \forall\,i\in\tilde\calV,j\in I^K(i).
\end{align}
The goal is to write the state $\bU^{\sharp}_{i,K}$ as a convex combination
of the low-order state $\bU^{\sharp,L}_i$ and high-order state
$\bU^{\sharp,H}_i$ so that $P_h^{\sharp}(\bu)\big|_K\in\text{conv}_{\bx\in
K}(\bu(\bx))$. Introducing $b^K_{ij}=m_i^K (m_{ij}^{K})^{-1}-\delta_{ij}$
and after some algebra we arrive at the identity:
\begin{align}
  \label{eq:element_wise_pi}
  \bU^{\sharp,H}_{i,K}\,-\,\bU^{\sharp,L}_{i,K}
  \;=\; \sum_{j} \kappa \bP_{ij}^K,
  \quad
  \bP_{ij}^K\;\eqq\;\frac{1}{\kappa m_i^K}\Big(b_{ij} \bR_j^K - b_{ji}
  \bR_i^K\Big),
\end{align}
where $\kappa=\text{card}\big(I^K(i)\big)^{-1}$. We now introduce limiter
coefficients $l_{ij}^K\in[0,1]$, with $l_{ij}^K= l_{ji}^K$, such that
\begin{align}
  \label{eq:element_wise_final}
  \bU^{\sharp}_{i,K} \;\eqq\; \bU^{\sharp,L}_{i,K} \;+\; \sum_{j}
  \kappa\,l^K_{ij} \bP_{ij}^K
  \quad\in\quad
  \conv_{\bx\in K}\big(\bu(\bx)\big).
\end{align}
Note that when $l_{ij}^K = 0$ we recover the low-order state,
$\bU^{\sharp,L}_{i,K}$, and when $l_{ij}^K = 1$ we recover the high-order
state, $\bU^{\sharp,H}_{i,K}$.  A detailed discussion on how the limiter
bounds are computed is outlined in Section~\ref{sec:details}. We collect
key properties of the projection in the following lemma.
\begin{lemma}\label{lem:mass_projection}
  Assume the limiting procedure returns symmetric coefficients $l_{ij}^K$
  and that $l_{ij}^K=1$ for the case of $\bu\in
  \mathcal{Q}_p^{\sharp}(\Pi_h)^m \cap \calA(\Omega)$. Then, the
  element-wise mass projection
  \begin{align*}
    P_h^{\sharp}\;:\; \calA(\Omega) \rightarrow
    \mathcal{Q}_p^{\sharp}(\Pi_h)^m \cap \calA(\Omega),
    \quad
    P_h^{\sharp}(\bu) = \sum_{i,K} \bU^{\sharp}_{i,K}\tilde
    \varphi_i^h\big|_K,
  \end{align*}
  is a mass conservative, invariant-domain preserving projection, meaning
  \begin{align*}
    \int_\Omega P_h^{\sharp}(\bu)\text{d}x\,=\,\int_\Omega \bu\text{d}x,
    \qquad
    P_h^{\sharp}(\bu)\big|_K \in \conv_{\bx\in K}\big(\bu(\bx)\big)
    \subset\mathcal{A}.
  \end{align*}
\end{lemma}
\begin{proof}
  The preliminary low- and high-order updates are conservative. This is an
  immediate consequence of the partition of unity property of the Lagrange
  basis $\tilde\varphi^h_{i}(\bx)$ on $K$. Summing \eqref{eq:element_wise_ri}
  shows:
  \begin{align*}
    \sum_{i,j}m_{ij}^K \bU^{\sharp,H}_{j,K} = \sum_{i}m_i^K
    \bU^{\sharp,L}_{i,K} = \int_K\bu(\bx)\text{d}x.
  \end{align*}
  Multiplying~\eqref{eq:element_wise_final} by $m_i^K$ and summing over
  $i\in\tilde\calV$ recovers the same property for the limited update:
  \begin{align*}
    \sum_{i}m_{i}^K \bU^{\sharp}_{i,K}
    = \sum_{i}\left(m_i^K \bU^{\sharp,L}_{i,K}\right) +
    \underbrace{\sum_{i,j} l_{ij}^K\Big(b_{ij} \bR_j^K - b_{ji}
    \bR_i^K\Big)}_{=\,0}
    = \int_K\bu(\bx)\text{d}x.
  \end{align*}
  Here, the last sum on the left hand side vanishes because $b_{ij} \bR_j^K
  - b_{ji} \bR_i^K$ is antisymmetric. Furthermore, the low-order update is
  an averaging procedure which preserves convex sets:
  \begin{align*}
    \bU^{\sharp,L}_{i,K} \;=\;
    \frac{1}{m_i^K}\int_K\bu(\bx)\tilde\varphi_i^h\text{d}x
    \;\in\;
    \conv_{\bx\in K}\big(\bu(\bx)\big).
  \end{align*}
  This ensures that the limiting procedure \eqref{eq:element_wise_final}
  has at least one solution, $l_{ij}^K = 0$. By
  construction we have that $\bU^{\sharp}_{i,K} \in \conv_{\bx\in
  K}\big(\bu(\bx)\big)$. It remains to show that $P_h^{\sharp}$ is a
  projection. Note that the preliminary high-order operator $P_h^{\sharp,
  H}$ is a projection, but the low-order order $P_h^{\sharp, L}$ is not.
  For a function $\bu(\bx)\in \mathcal{Q}_p^{\sharp}(\Pi_h)^m$ it follows
  that the high-order mass projection gives $\sum_{i,K}
  \bU^{\sharp,H}_{i,K}\tilde \varphi_i^h\big|_K(\bx)=\bu(\bx) \in
  \conv_{\bx\in K}\big(\bu(\bx)\big)$. This implies that we can choose
  $l_{ij}^K=1$ throughout and---provided the limiting procedure is optimal
  in this regard---we conclude that $\bU^{\sharp}_{i,K} =
  \bU^{\sharp,H}_{i,K}$, which implies $P_h^{\sharp}\bu\equiv\bu$.
\end{proof}

\subsection{Nodal averaging projection}\label{sec:nodal-averaging}
We now define the nodal averaging projection $P_h^{\text{avg.}}$. Let
$\bu_h^{\sharp} = \sum_{i,K} \bU^{\sharp}_{i,K}\tilde\varphi_i^h\big|_K
\in\mathcal{Q}_p^{\sharp}(\Pi_h)^m$ be arbitrary. We then set
\begin{gather*}
  P_h^{\text{avg.}}\;:\; \mathcal{Q}_p^{\sharp}(\Pi_h)^m \cap \calA(\Omega)
  \quad\rightarrow\quad
  \tilde{\mathcal{Q}}_{p} \cap \calA(\Omega),
  \\[0.25em]
  P_h^{\text{avg.}}(\bu_h^{\sharp}) \eqq \sum_{i}
  \tilde{\bU}_{i}\tilde\varphi_i^h, \quad\text{with}\quad
  \tilde{\bU}_{i}\,\eqq\, \frac{1}{\tilde m_i}\sum_K
  m_i^K\,\bU^{\sharp}_{i,K}.
\end{gather*}
From this definition and recalling that $\tilde m_i=\sum_{K}m_i^K$ it
follows immediately that $P_h^{\text{avg.}}$ is a projection. The
projection consists of a simple averaging procedure; it is thus mass
conservative and preserves the invariant domain:
\begin{align*}
  \int_\Omega P_h^{\text{avg.}}(\bu^\sharp_h)\text{d}x\,=\,\int_\Omega
  \bu^\sharp_h\text{d}x,
  \qquad
  P_h^{\text{avg.}}(\bu^\sharp_h)\big|_K \in \conv_{\bx\in
  K}\big(\bu^\sharp_h(\bx)\big)\subset\mathcal{A}.
\end{align*}

\subsection{Mass redistribution}\label{sec:mass-redistribution}
As a final step in our mass projection we now need to enforce hanging node
constraints $\polA\polC_p(\Pi_h)$ on the intermediate result of the nodal
averaging projection via the operator $P_h^{\text{red.}}$. To this end let
$\tilde\bu_h = \sum_{i} \tilde\bU_{i}\tilde\varphi_i^h
\in\tilde{\mathcal{Q}}_p(\Pi_h)^m$ be arbitrary. We recall that
$\big\{\tilde\varphi_i^h(\vec x)\,:\,i\in\tilde\calV\big\}$ is the Lagrange
basis of $\tilde{\mathcal{Q}}_{p}(\Pi_h)$, which is defined in
\eqref{eq:space_sadreality}. We adopt the notation
$j\not\in\polA\polC$ to denote any index $j\in\tilde{\mathcal{V}}$
corresponding to an unconstrained degree of freedom, i.\,e.,
$(j,\ast)\not\in\polA\polC_p(\Pi_h)$.

Let $\big(j,\big\{c_i^j\,:\, i\in\tilde{\mathcal{V}}\big\}\big)$ be an
entry in the affine constraint set $\polA\polC_p(\Pi_h)$ and set
\begin{align*}
  \bD_j\;\eqq\;\tilde\bU_j-\sum_{i\in\tilde{\mathcal{V}}}c_i^j\tilde\bU_{i}.
\end{align*}
$\bD_j$ quantifies the mismatch between the preliminary state $\tilde\bU_j$
of a constrained degree of freedom $j$ and its value according to the
hanging node constraint $\polA\polC_p(\Pi_h)$. Similar to the procedure
outlined in Section~\ref{sec:element-wise}, we introduce again a low-order
and a high-order operator:
\begin{align*}
  \begin{aligned}
    P_h^{\text{red.},L}\;:\;
    \mathcal{Q}_p^{\sharp}(\Pi_h)^m \cap \calA(\Omega) &\rightarrow
    \mathcal{Q}_p(\Pi_h)^m,
    &\quad
    P_h^{\text{red.},L}(\tilde\bu_h) &= \sum_{i} \bU^{L}_{i}\tilde
    \varphi_i^h, \\
    P_h^{\text{red.},H}\;:\;
    \mathcal{Q}_p^{\sharp}(\Pi_h)^m \cap \calA(\Omega) &\rightarrow
    \mathcal{Q}_p(\Pi_h)^m \cap \calA(\Omega),
    &\quad
    P_h^{\text{red.},H}(\tilde\bu_h) &= \sum_{i} \bU^{H}_{i}\tilde
    \varphi_i^h,
  \end{aligned}
\end{align*}
where the states $\bU^{L}_i$,  $\bU^{H}_i\in\polR^m$, are constructed as
follows for $i\not\in\polA\polC$:
\begin{align}
  m_i\,&\eqq\,\tilde m_i + \sum_{j\in I^{\polA\polC}(i)} c^j_i\tilde
  m_j,\notag
  \\
  \label{eq:low_order_red}
  \bU^{L}_i\;&\eqq\; \frac
  {\tilde m_i\tilde\bU_i \;+\; \sum_{j\in I^{\polA\polC}(i)}c^j_i\tilde
  m_j\tilde\bU_j}{m_i},
  \\
  \label{eq:high_order_red}
  \bU^{H}_i\;&\eqq\; \frac
  {\tilde m_i\tilde\bU_i \;+\; \sum_{j\in I^{\polA\polC}(i)}c^j_i\tilde
  m_j\big(\tilde\bU_i+\bD_j\big)}{m_i}.
\end{align}
Note that the mass lumping strategy constructing $m_i$, which includes the contribution from adjacent constrained degrees of freedom, is a consistent choice for the Gau\ss{}--Lobatto basis, see, e.g.,~\cite[Sec.~5.3]{Kormann2016time}. For constrained degrees of freedom $i\in\polA\polC$, we apply
Definition~\ref{def:algebraic-constraints} and simply set
\begin{align*}
  \bU^{\ast}_i\;&\eqq\; \sum_{j\in\tilde{\mathcal{V}}}c_j^i\bU^{\ast}_j,
  \qquad\ast=L,H.
\end{align*}

\begin{lemma}\label{lem:mass_redistribution}
  The low-order update is conservative and preserves an invariant domain,
  viz. $\bU^{L}_i\;\in\; \conv\big\{\tilde\bU_i,\conv_{j\in
  I^{\polA\polC}(i)}\big(\tilde\bU_j)\big\}$. The high-order update
  is conservative and describes a projection into $\mathcal{Q}_p(\Pi_h)^m$.
\end{lemma}
\begin{proof}
  For $\tilde\bu_h\in\mathcal{Q}_p(\Pi_h)^m$ affine constraints are already
  satisfied and $\bD_j=\vec0$. In this case \eqref{eq:high_order_red}
  reduces to the identity operation, $\bU^{H}_i=\tilde\bU_i$, implying that
  $P_h^{\text{red.},H}$ is indeed a projection. Rewriting the low-order
  update yields:
  \begin{align*}
    \bU^{L}_i\,=\,
    \frac{\tilde m_i}{m_i}\tilde\bU_i + \sum_{j\in I^{\polA\polC}(i)}
    \frac{c^j_i\tilde m_j}{m_i}\tilde\bU_j
    \;\in\; \conv\big\{\tilde\bU_i,\conv_{j\in
    I^{\polA\polC}(i)}\big(\tilde\bU_j)\big\}.
  \end{align*}
  Regarding mass conservation we observe that:
  \begin{align*}
    \sum_{i\in\tilde{\mathcal{V}}}\tilde m_i\bU_i^{L}
    \;&=\;
    \sum_{i\not\in\polA\polC}m_i\bU_i^{L}
    \;=\;
     \sum_{i\not\in\polA\polC}\tilde m_i\tilde\bU_i
    +\sum_{j\in\polA\polC}\underbrace{\Big(\sum_{i\in\tilde{\mathcal{V}}}
    c_i^j\Big)}_{=\,1}\tilde m_j\tilde\bU_j
    \;=\; \sum_{i\in\tilde{\mathcal{V}}}\tilde m_i\tilde\bU_i,
    \intertext{and}
    \sum_{i\in\tilde{\mathcal{V}}}\tilde m_i\bU_i^{H}
    \;&=\;
    \sum_{i\not\in\polA\polC}m_i\bU_i^{H}
    \\
    \;&=\;
     \sum_{i\not\in\polA\polC}\tilde m_i\tilde\bU_i
    +\sum_{i\not\in\polA\polC}\sum_{j\in\tilde{\mathcal{V}}}
    c_i^j \tilde m_j \big(\tilde\bU_i+\tilde\bU_j-
    \sum_{k\in\tilde{\mathcal{V}}}c^j_k\tilde\bU_k\big)
    \\
    \;&=\;
     \sum_{i\not\in\polA\polC}\tilde m_i\tilde\bU_i
    +\sum_{j\in\polA\polC}\underbrace{\Big(\sum_{i\in\tilde{\mathcal{V}}}
    c_i^j\Big)}_{=\,1}\tilde m_j\tilde\bU_j
    \\
    &\qquad\qquad\qquad\qquad
    +\sum_{j\in\tilde{\mathcal{V}}} \tilde m_j
    \underbrace{\Big(\sum_{i\not\in\polA\polC} c_i^j\tilde\bU_i -
    \sum_{k\not\in\polA\polC}\big(\sum_{i\not\in\polA\polC} c_i^j\big)
    c^j_k\tilde\bU_k\Big)}_{=\,\vec 0}
    \\
    \;&=\; \sum_{i\in\tilde{\mathcal{V}}}\tilde m_i\tilde\bU_i.
  \end{align*}
\end{proof}
Similarly to the element-wise projection we now rearrange the low and high
order update into
\begin{align*}
  \bU^{H}_i - \bU^{L}_i
  \;=\;
  \sum_{j\in\polA\polC} \kappa \bP_i^j,
  \quad\text{with }
  \bP_i^j\eqq\frac{c_i^j\tilde m_j}{m_i\,\kappa}
  \Big(\tilde\bU_i-\sum_{k\in\tilde{\mathcal{V}}}c_k^j\tilde\bU_k\Big),
\end{align*}
where $\kappa=\text{card}\big(I^K(i)\big)^{-1}$. We introduce convex
limiter bounds $l_i^j\in[0,1]$ such that for $i\not\in\polA\polC$:
\begin{align*}
  \bU^{L}_{i} + \,l_i^j \bP_i^j
  \quad\in\quad
  \conv\big\{\tilde\bU_i,\conv_{j\in
  I^{\polA\polC}(i)}\big(\tilde\bU_j)\big\}.
\end{align*}
Then setting $l^j=\min_{i\in\tilde{\mathcal{V}}}l_i^j$ we form a limited
redistribution update:
\begin{align}
  \label{eq:red_final}
  \bU_{i} \eqq \bU^{L}_{i} + \sum_{j\in\polA\polC} \kappa l^j \bP_i^j
  \quad\in\quad
  \conv\big\{\tilde\bU_i,\conv_{j\in
  I^{\polA\polC}(i)}\big(\tilde\bU_j)\big\}.
\end{align}
\begin{lemma}\label{lem:conservative}
  The convex limiting procedure \eqref{eq:red_final} is conservative.
\end{lemma}
\begin{proof}
  We verify that
  \begin{align*}
    \sum_{i\not\in\polA\polC} m_i\sum_{j\in\polA\polC} \kappa l^j \bP_i^j
    \;&=\;
    \sum_{i\not\in\polA\polC}\sum_{j\in\polA\polC} l^j c_i^j\tilde m_j
    \Big(\tilde\bU_i-\sum_{k\in\tilde{\mathcal{V}}}c_k^j\tilde\bU_k\Big)
    \\
    \;&=\;
    \sum_{j\in\polA\polC} l^j\tilde m_j
    \Big(\sum_{i\not\in\polA\polC} c_i^j \tilde\bU_i -
    \sum_{k\not\in\polA\polC}c_k^j
    \big(\sum_{i\in\tilde{\mathcal{V}}} c_i^j\big)
    \tilde\bU_k\Big)
    \\
    \;&=\;\vec0.
  \end{align*}
\end{proof}


\section{Implementational details}
\label{sec:details}
We now describe a concrete implementation strategy and algorithm for the mass
projections introduced in Section~\ref{sec:mass_projection} for a mesh
adaptation cycle. In addition, we comment on how to expand this strategy
for mesh transfer between unrelated meshes that occurs, for example, during
re-meshing. The strategies has been implemented in
\texttt{ryujin}~\cite{maier21b}, which is based on the \texttt{deal.II}
finite element library~\citep{dealIIcanonical, dealII96}.

\subsection{State projection during mesh adaptation}
A mesh adaptation cycle consists of first marking a subset of cells for
coarsening or refinement with a suitable marking strategy, then resolving
some algebraic constraints (such as imposing a maximal level difference
between cells, see Section~\ref{sec:preliminaries}) and finally adjusting
the mesh. \texttt{deal.II} supports this operation on fully distributed
meshes where only a part of the mesh is stored for each participating MPI
process~\cite{dealIIcanonical}, and where the parallel partition of the
distributed mesh changes with each mesh adaptation cycle. \texttt{deal.II}
solves the question how to transfer numerical solutions from the old to the
new mesh by first attaching additional data to each relevant cell that
fully describes the local state, then, if necessary, transferring this data
to the correct MPI rank, and, finally, reconstructing the state vector from
this cell local data. \texttt{deal.II} contains an implementation of this
strategy in form of the \texttt{SolutionTransfer} class that performs plain
Lagrange interpolation. We now describe a modification of this approach
that uses the invariant-domain preserving mass projection outlined in
Section~\ref{sec:mass_projection}. The fact that the mesh transfer
operation projects from one finite element space into a closely related
finite element space associated with a mesh where only a small number of
cells have been coarsened or refined allows for some important shortcuts.
Most notably, for a given cell $K$ the element-wise mass projection
$P_h^{\sharp}\big|_{K}\,:\, \mathcal{Q}_p^{\sharp}(K)^m \rightarrow
\mathcal{Q}_p^{\sharp}(K)^m$
reduces to the identity operation for the case that the cell $K$ is neither
refined nor coarsened. Another practical issue to solve is that of how to
construct the bounds for the convex limiting step.

The distributed mesh class in \texttt{deal.II} provides a callback
\texttt{register\_data\_attach} for computing and attaching said additional
data. The callback takes two arguments: the current cell $K$ to operate on
and a \emph{cell status} $S$ indicating whether the
\texttt{cell\_will\_persist}, \texttt{cell\_will\_be\_refined}, or
\texttt{children\_will\_be\_coarsened}.
Algorithm~\ref{alg:register_data_attach} (on
page~\pageref{alg:register_data_attach} in Appendix~\ref{app:algorithms})
summarizes in pseudo code the implementation found in \texttt{ryujin}. The
call back is invoked at the beginning of the mesh adaptation cycle for
collecting additional data associated with a cell. At the end of the mesh
adaptation cycle the data is used to reconstruct the state vector. For
this, we first invoke another call back, \texttt{notify\_ready\_to\_unpack}
to repopulate the state vector and to perform the nodal averaging (see
Section~\ref{sec:nodal-averaging}). We summarize the implemention of the
second call back in Algorithm~\ref{alg:notify_ready_to_unpack} (on
page~\pageref{alg:notify_ready_to_unpack} in
Appendix~\ref{app:algorithms}). We note that the limiting step for the
refined cell that is performend in
Algorithm~\ref{alg:notify_ready_to_unpack} can be avoided when the
(combined) finite element space of the new child cells already contain the
former coarse cell shape functions.

\subsection{Mass redistribution}
After the state projection we are left with a state vector $\tilde\bU_i$,
where $\tilde\bu_h(\vec x) = \sum_{i\in\tilde\calV}
\tilde\bU_i\tilde\varphi_i^h(\vec x)$ is not compatible with the hanging
node constraints $\polA\polC_p(\Pi_h)$. Following the strategy outlined in
Section~\ref{sec:mass-redistribution} we now construct modified masses
$m_i$ and a modified state vector $\bU_i$ such that
\begin{align*}
  \sum_{i\in\tilde{\mathcal{V}}}\tilde m_i\tilde\bU_i
  \,=\,
  \sum_{i\in\tilde{\mathcal{V}}}\tilde m_i\bU_i
  \,=\,
  \sum_{i\not\in\polA\polC}m_i\bU_i
\end{align*}
holds true, the states $\bU_i$ remain admissible, and $\polA\polC_p(\Pi_h)$
are enforced on $\bu_h(\vec x)=\sum_{i\in\tilde\calV}
\bU_i\tilde\varphi_i^h(\vec x)$; see Algorithm~\ref{alg:redistribute} (on
page~\pageref{alg:redistribute} in Appendix~\ref{app:algorithms}).

\subsection{Smoothness indicators}\label{sec:smoothness}
For the numerical tests we use a simple smoothness indicator computed on a
number of quantities of interest. To this end, let $\beta_{ij}$ denote the
entries of the discrete Laplace matrix that has been assembled subject to
the hanging node constraints $\polA\polC_p(\Pi_h)$ \cite{dealIIcanonical,Shephard1984linear}, that is,
\begin{align*}
  \beta_{ij} =
  \sum_{\substack{m\,\in\,\{i\}\cup I^{\polA\polC}(i) \\
  n\,\in\, \{j\}\cup I^{\polA\polC}(j)}}
  c_i^m\,c_j^n\,
  \int_\Omega\tilde\varphi^h_m(\vec x)\tilde\varphi^h_n(\vec x)\text{d}x
\end{align*}
for $i,j\not\in\polA\polC_p(\Pi_h)$. Here, we have used---by slight
abuse of notation--- the convention $c_i^i=c_j^j=1$. If one of the
indices $i$ or $j$ is constrained, we set $\beta_{ij}=0$. Let
$I(i)=\{j\in\tilde\calV\,:\,\beta_{ij}\not=0\}$.

Given $\{\bQ_i\}_{i\in\tilde\calV}\subset\mathbb{R}^{N}$, a vector of
$N$ quantities of interest that shall be used for constructing a smoothness
indicator, we set:
\begin{gather}
  \label{eq:indicator_numerator_denominator}
  \alpha_i\eqq\sum_{\nu=1}^{N}\frac{|n_i^\nu|}{(1-\kappa)\,d_i^\nu\,+\,
  \kappa\max_{i}d_i^\nu}, \qquad\text{where}
  \\[0.25em]
  \notag
  n_i^{\nu}\eqq\sum_{j\in I(i)}\beta_{ij}(Q_j^\nu-Q_i^\nu),\qquad
  d_i^{\nu}\eqq\sum_{j\in I(i)}|\beta_{ij}|\,\left|Q_j^\nu-Q_i^\nu\right|,
\end{gather}
and where $\kappa\in[0,1]$ is a suitably chosen scaling parameter. We
\emph{extend} the indicator by computing the maximum over the stencil a
selected number of times. For given $R\ge0$, set $\alpha_i^{(0)}=\alpha_i$
and then iterate:
\begin{align*}
  \alpha_i^{(k)} = \max\big(\{\alpha_i^{(k-1)}\}\cup\{\alpha_j^{(k-1)}\,:\,
  j\in I(i)\big\}\big),
  \quad\text{for }k=1,\ldots,R.
\end{align*}
We note here that the indicator defined above is not the most optimal, but it is sufficient for testing the novel projection technique. Developing a novel indicator is out of the scope of this work.

\subsubsection{Threshold marking strategy}
As a final step, the smoothness indicator $\{\alpha_i^{(R)}\}$ is
transformed into a cell-wise indicator:
\begin{align*}
  \alpha_K\;\eqq
  \frac{1}{\#\big\{j\in\tilde\calV\,:\, m_j^K\not=0\big\}}
  \sum_{\substack{j\in\tilde\calV \\ m_j^K\not=0}} \alpha_j^{(R)}.
\end{align*}
Given two configurable parameters,
$0\le\alpha_K^{\text{coar}}\le\alpha_K^{\text{ref}}$, we then mark all
cells with an indicator above a configured threshold for refinement,
$\alpha_K\ge\alpha_K^{\text{ref}}$, and all cells with an indicator below a
configured threshold for coarsening, $\alpha_K\le\alpha_K^{\text{coar}}$.
We have observed numerically that for problems with localized features, a value of $\alpha_K^{\text{ref}}\approx0.2$ is sufficient. For problems with multiple regions of interest, $\alpha_K^{\text{ref}}\approx0.05$ is sufficient.


\section{Numerical illustrations}
\label{sec:illustrations}
We now illustrate the proposed mass-conservative and invariant-domain
preserving projection technique with a suite of simulations to highlight its robustness.
For the sake of completeness, we recall each hyperbolic system of interest and its admissible set. 

\subsection{Preliminaries}
All the simulations are performed with the high performance code, \texttt{ryujin}~\cite{maier21b}. 
The time stepping is done with an invariant-domain preserving three stage, three step Runge--Kutta method as described in~\citet{ern2022invariant}. All tests are performed with the CFL number 0.9. 
The smoothness indicator described in \S\ref{sec:smoothness} is used for the adaptive mesh refinement, which is performed at each simulation cycle. 
To verify mass conservation of our scheme, we measure the relative error of the total mass (at the every simulation cycle) relative to the total mass at the initial time. We observe that in every test with all-slip boundary conditions (that is, no information going into or out of the computational domain), this relative errors stays at machine precision. 

\subsection{Shallow water equations}\label{rem:shallow_water}
The shallow water equations are often used in modeling surface water
waves under the action of gravity. Let
$\bu\eqq(\waterh,\bq)^{\mathsf{T}}$ where $\waterh$ denotes the water
depth, $\bq$ the momentum and $\bv\eqq\bq/\waterh$ the velocity. The
hyperbolic flux is defined by: $\polf(\bu)\eqq(\bv, \bv\otimes\bq +
\frac{1}{2}g\waterh^2\polI_d)^{\mathsf{T}}$ with $g = \SI{9.81}{m
s^{-2}}$ the gravitational constant. The admissible set for the
shallow water equations is given by:
$\calA\eqq\{\bu\eqq(\waterh,\bq)^{\mathsf{T}} : \waterh > 0\}$.
The invariant-domain preserving spatial approximation technique adopted here for the shallow water equations is outlined in~\cite{guermond2025high}.

\subsubsection{Dam break over three conical obstacles}
We test the ability of the method to handle dry states. We consider the topography configuration introduced in~\cite{umetsu1986finite}, consisting of three conical obstacles in a channel. The initial setup is a dam break configuration over a dry bottom: $\mathsf{h}_0(\bx)=\SI{1.875}{m}$ for $x_1 \leq \SI{16}{m}$ and $0$ elsewhere. The initial discharge, $\mathbf{q}_0(\bx) = \bm{0}$ is zero everywhere. The domain is $D = (0, \SI{75}{m})\times(0, \SI{30}{m})$ with slip boundary conditions. The final time is set to $T = \SI{8}{s}$. We show the results in Figure~\ref{fig:swe-dam-break}. 
\begin{figure}[t]
    \centering
    \includegraphics[clip,trim = {0, 0, 0, 10},width=0.59\linewidth]{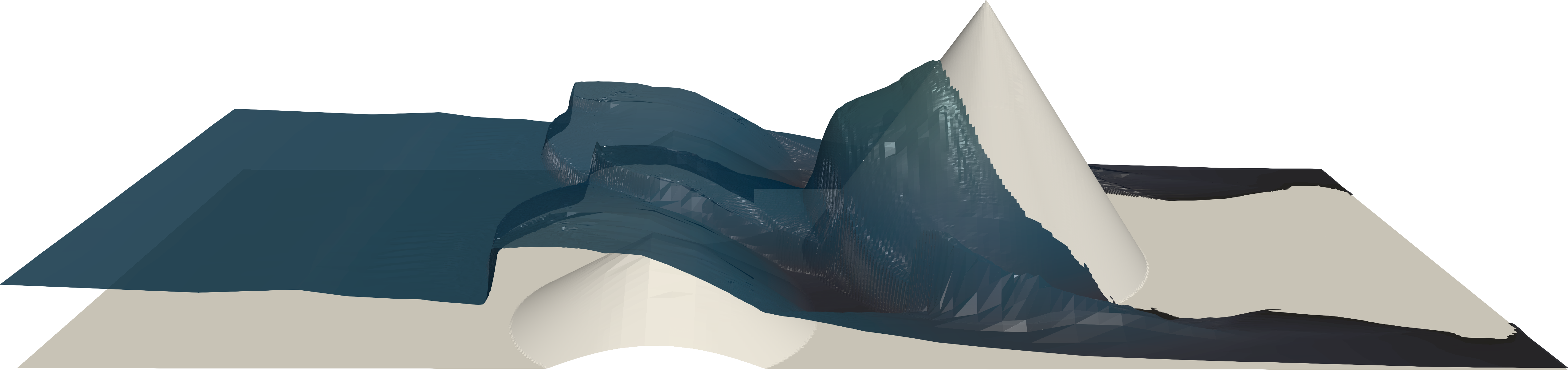}
    \includegraphics[clip,trim = {0, 0, 0, 0},width=0.395\linewidth]{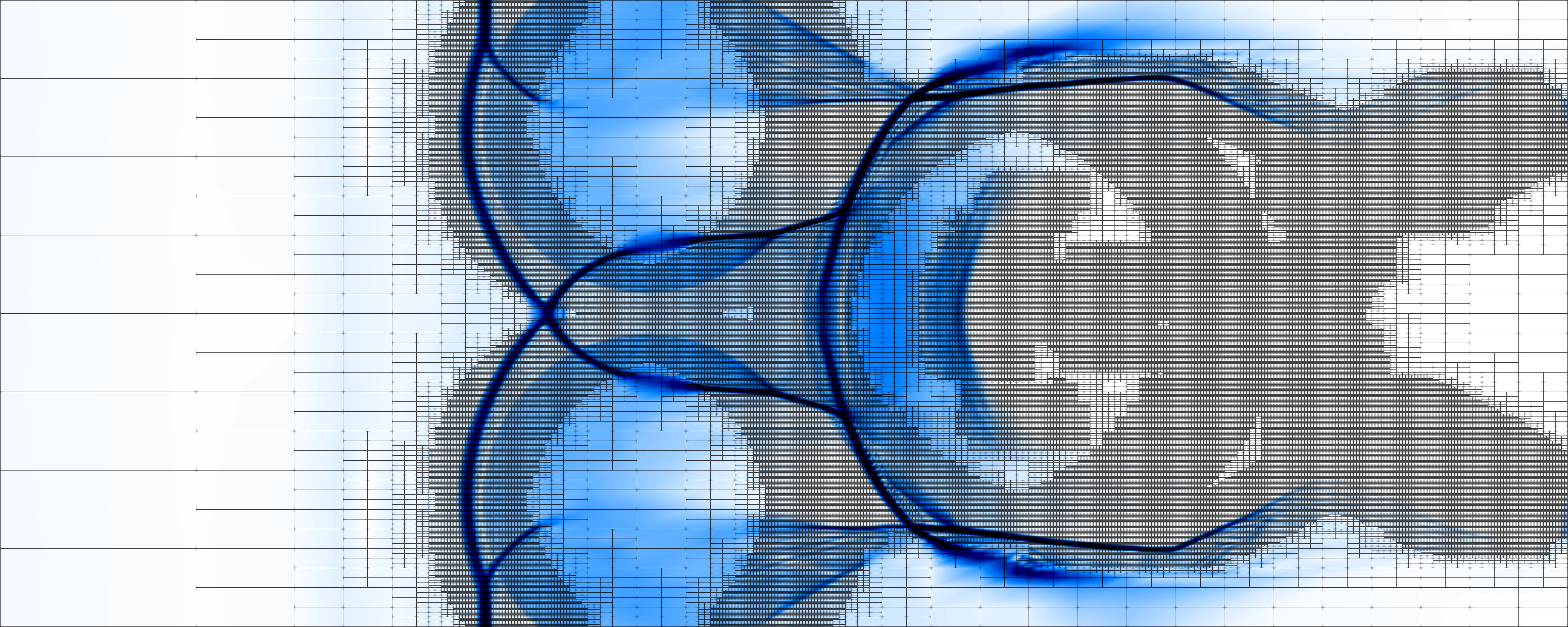}
    \caption{Dam break -- Water elevation and topography (left); water depth numerical schlieren with computational mesh (right) at $T = \SI{8}{s}$. }
    \label{fig:swe-dam-break}
    \centering
    \def\addlegendimage{\csname pgfplots@addlegendimage\endcsname}
    \begin{centering}
        \begin{tikzpicture}
        \begin{customlegend}[legend columns=3,legend style={align=left,draw=white!15!black, anchor=west},legend entries={Conventional, IDP}]
        \addlegendimage{black, thick}
        \addlegendimage{blue, ultra thick}   
        \end{customlegend}
     \end{tikzpicture}\end{centering}\\
    \centering
    \begin{tikzpicture}[font=\normalsize, trim axis left, trim axis right]
\begin{axis}[%
xtick pos=left, ytick pos=left, ymode=log, ylabel={$\Delta t$ [s]}, xlabel={Simulation time [s]},  every axis plot/.append style={ultra thick}, width=0.65\linewidth,height=0.3\linewidth
]
\addplot[black, thick] table[x expr =\thisrow{index}*0.01354, y=dt, col sep=space] {swe-broken-dt.dat};
\addplot[blue, ultra thick] table[x expr =\thisrow{index}*0.0134, y=dt, col sep=space] {swe-correct-dt.dat};
\end{axis}
\end{tikzpicture}
    \caption{Dam break -- Comparison of time-step with full mass-conservative IDP projection (blue) vs mass-conservative non-IDP projection.}
    \label{fig:time-step-comp}
\end{figure}
We note that the projection technique outlined above does not take into account source terms and thus we can not expect the well-balancing property to hold. However, this is out of the scope of the present contribution. We see the IDP projection technique is able to handle the complex wetting/drying process.

\subsubsection{Robustness check}
We perform the same test as above with no limiting during the projection process. That is, the projection is only mass-conservative but not invariant-domain preserving. We observe numerically that the projection without limiting leads to a degradation of the time step -- particularly in the ``drying'' process. In Figure~\ref{fig:time-step-comp}, we plot the time-step $\Delta t$ over simulation time for each simulation. We see that the simulation without AMR limiting leads to a time-step on the order of $\calO(10^{-10})$ (necessitating termination), demonstrating the need for the AMR projection limiting procedure. In contrast, with the proposed IDP AMR scheme, we find that $\Delta t \sim\cal{O}(10^{-2})$ throughout the entire simulation.

\subsection{Compressible Euler Equations}
The Euler equations are used in modeling compressible gas dynamics. Let
$\bu\eqq(\rho,\bbm, E)^{\mathsf{T}}$ with $\rho$ denoting the material
density, $\bbm$ the momentum and $E$ the total energy. Let $\bv\eqq\bbm /
\rho$ be the velocity. The hyperbolic flux is defined by:
$\polf(\bu)\eqq(\bv, \bv\otimes\bbm + p(\bu)\polI_d, \bv(E +
p(\bu)))^{\mathsf{T}}$ where $p(\bu)$ is the pressure law defined by an
equation of state. The constraints that make up the admissible set for
the Euler equations are often motivated by the underlying equation of
state. Assuming that pressure is non-negative, we consider the following
generic admissible set for the Euler equations:
$\calA\eqq\{\bu\eqq(\rho,\bbm, E)^{\mathsf{T}} : \rho > 0,\,e(\bu) >
0\}$ where $e(\bu)\eqq \rho^{-1}(E  - \frac12 \rho\|\bv\|^2)$ is the
specific internal energy. The invariant-domain preserving spatial approximation techniques for ideal gas and general equations of state are outlined in~\citep{guermond2018second} and \citep{clayton2023robust}, respectively.

\subsubsection{Astrophysical jet (ideal gas)}
We now consider a high Mach astrophysical jet benchmark introduced in~\cite{ha2005numerical} and reproduced in~\cite[Sec.~3.4]{zhang2011positivity},~\cite[Sec.~5.2.3]{rueda2022subcell} and many others.
To highlight the robustness of the proposed projection method, we modify the initial conditions slightly so that the astrophysical jet has a Mach number of 4000. 
More specifically, let the equation of state for the compressible Euler equations be the ideal gas law with $\gamma = \frac{5}{3}$. Let $\bw(\bx, t) \eqq (\rho, v_1, v_2, p)^{\mathsf{T}}$ denote the primitive state. At initial time, the state is ambient with $\bw_0(\bx) = (0.5, 0, 0, 0.4127)^{\mathsf{T}}$. Let $\rho_{\text{jet}} = 5$ and $p_{\text{jet}} = 0.4127$. Then, we set $v_{\text{jet}} = 4000 \sqrt{\gamma \frac{p_{\text{jet}}}{\rho_{\text{jet}}}}\approx 1483.59922710503$. Then, the jet with width $\SI{0.1}{m}$ is defined via the state $\bw_0(\bx) = (\rho_{\text{jet}}, v_{\text{jet}}, 0, p_{\text{jet}})^{\mathsf{T}}$ and is introduced into the domain at $(\SI{0}{m}, \SI{0}{m})$ via a Dirichlet boundary condition.

The computational domain is $D = (\SI{0}{m}, \SI{0.2}{m})\times(-\SI{0.1}{m}, \SI{0.1}{m})$ with a Dirichlet boundary on the left and free flow boundary conditions elsewhere.
The final time is set to $T = \SI{2e-4}{s}$. 
We report in Figure~\ref{fig:ideal-jet} the result at $t = \SI{1e-4}{s}$ showing the density in logarithmic scale and the computational mesh above the $x_2$-centerline. 
We see that even though the problem is highly kinetic energy dominated, the projection technique did not violate the invariant domain properties (\ie the internal energy remains positive).
\begin{figure}
    \centering
    \includegraphics[clip,trim = {0, 0, 0, 0},width=0.497\linewidth]{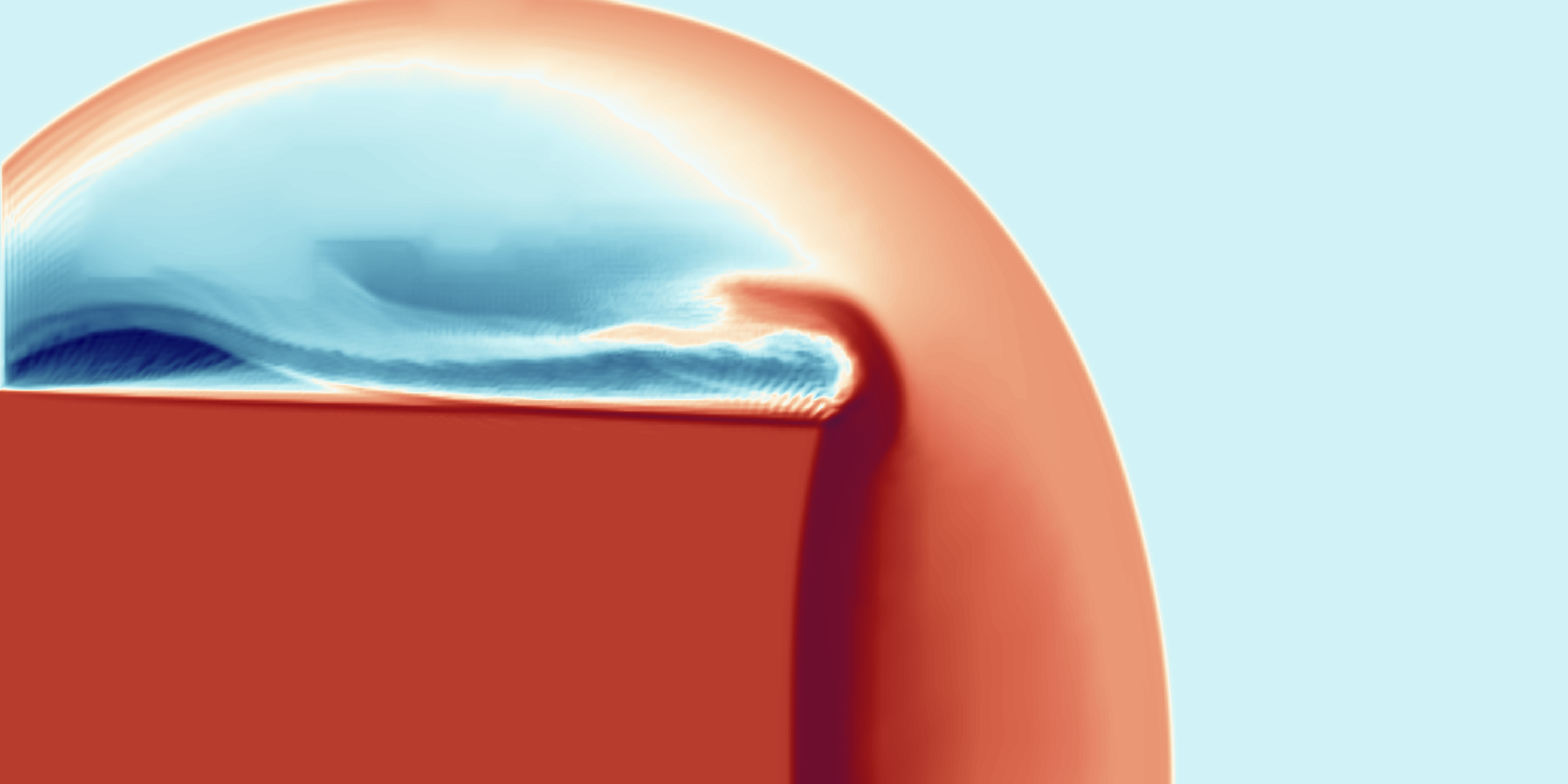}
    \hspace{-2mm}
    \includegraphics[clip,trim = {0, 0, 0, 0},width=0.497\linewidth]{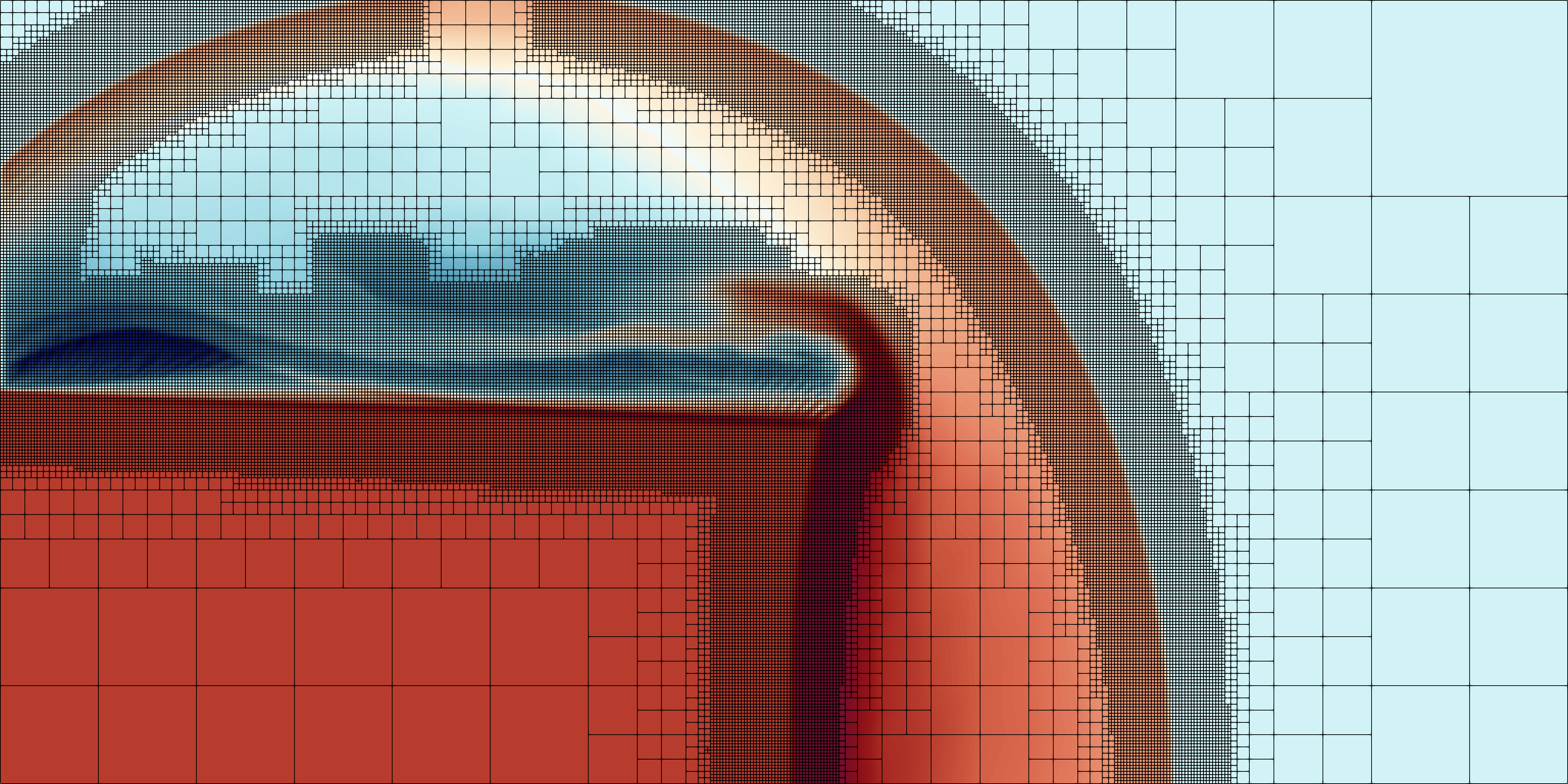}
    \caption{Astrophysical jet -- Density in logarithmic scale (left); density and computational mesh (right) at $t = \SI{1e-4}{s}$ above the $x_2 = 0$ centerline.}
    \label{fig:ideal-jet}

    \includegraphics[clip,trim = {0, 0, 0, 0},width=0.49\linewidth]{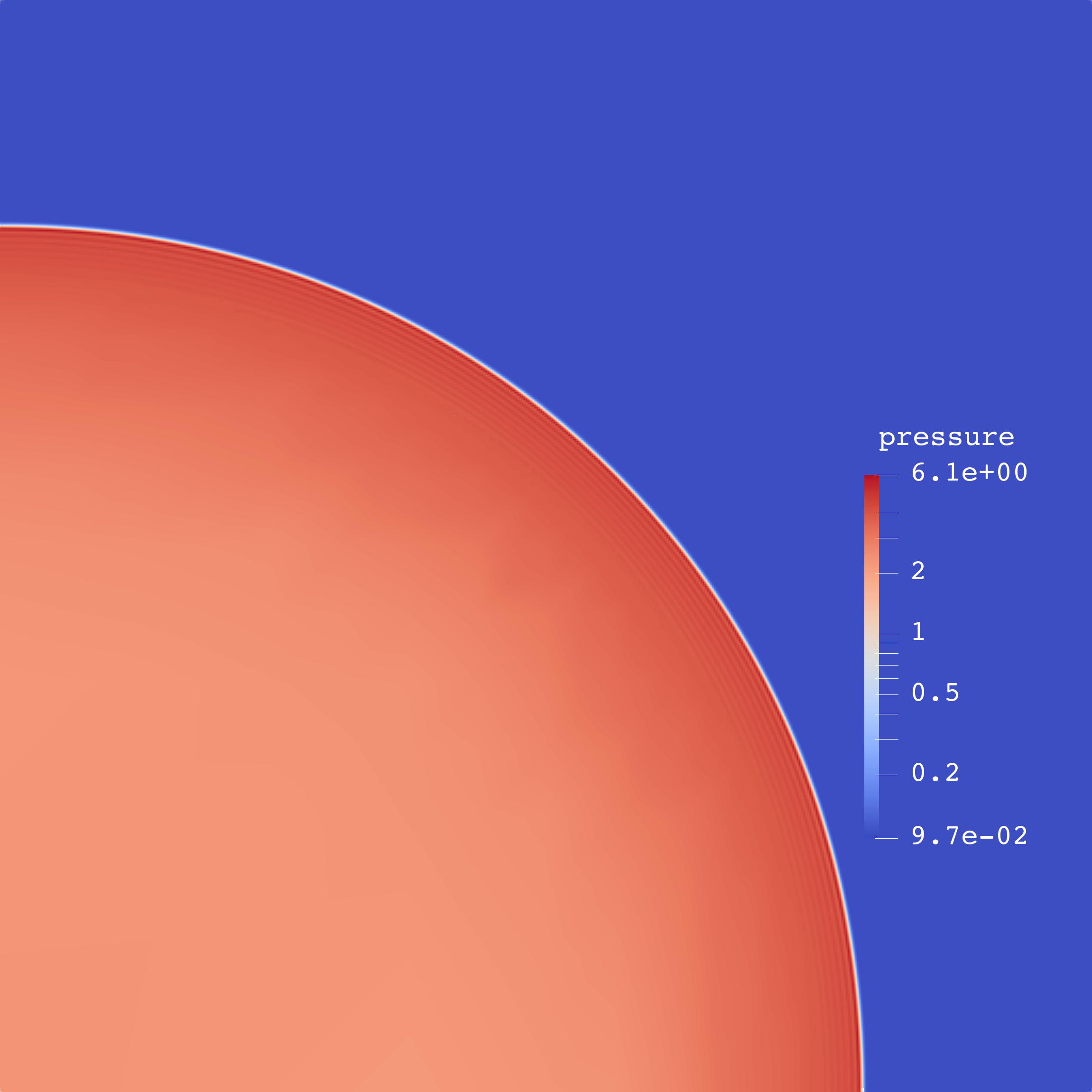}
    \includegraphics[clip,trim = {0, 0, 0, 0},width=0.49\linewidth]{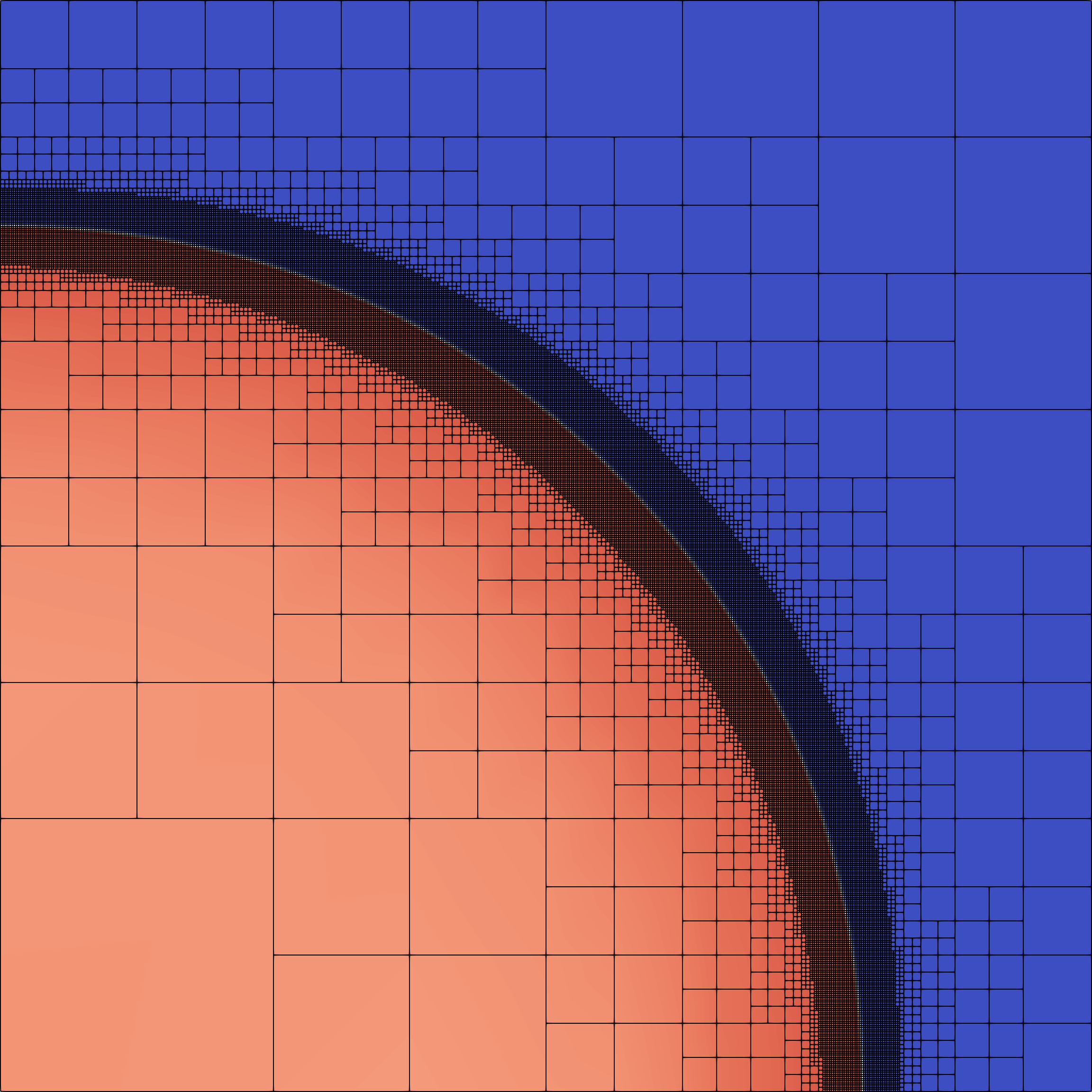}
    \caption{Sedov-like blast wave -- Pressure in logarithmic scale (left) and pressure with computational mesh (right) at $t = \SI{0.05}{s}$ with the Jones-Wilkins-Lee equation of state. }
    \label{fig:jwl-sedov}
\end{figure}

\subsubsection{Sedov-like explosion with Jones-Wilkins-Lee EOS}
To demonstrate the robustness of the proposed method with respect to
general equations of state, we perform a 2D Sedov-like blast wave benchmark
using the Jones-Wilkins-Lee equation of state (\cite{lee_1968,
segletes2018examination}). The equation of state parameters are given in
Table~\ref{tab:jwl_eos}.
\begin{table}[ht!]
  \centering
  \begin{tabular}[b]{lllllll}
      \toprule
      $A$ & $B$ & $R_1$ & $R_2$ & $\omega$ & $\rho_0$ & $e_0$
      \\[0.3em]
      \SI{6.321e3}{}   & \SI{-4.472}{}     & 11.3 & 1.13 & 0.8938 & \SI{1}{} & \SI{0}{}\\
      \bottomrule
    \end{tabular}
  \caption{Jones-Wilkins-Lee equation of state parameters.}
  \label{tab:jwl_eos}
\end{table}
The setup is as follows. The ambient state is set by $\bw_0(\bx) =
(\SI{1}{kg~m^{-3}}, 0, 0, \SI{0.1}{Pa})^{\mathsf{T}}$. The higher-pressure
region $\bw_0(\bx) = (\SI{1}{kg~m^{-3}}, 0, 0, \SI{100}{Pa})^{\mathsf{T}}$
is initialized in $\|\bx\|\leq \SI{0.05}{m}$. The computational domain is
$D = (\SI{0}{m}, \SI{0.4}{m})^2$. We set the final time to $T =
\SI{0.05}{s}$. We show in Figure~\ref{fig:jwl-sedov} the pressure (in
logarithmic scale) with the computational mesh at the final time. We see
that even for general equations of state, the projection technique does not
produce negative internal energies in the presence of large jumps in
pressure.

\paragraph{Performance comparison.}
A short remark on the performance of the mesh adaptation is in order. For
comparison, we redo the computation of the Sedov-like explosion shown in
Figure~\ref{fig:jwl-sedov} with a uniformly refined grid on the finest
refinement level ($l=9$ with 263\,169 degrees of freedom per component)
that we allowed during the mesh adaptation. The simulation ran for 114
seconds on 8 cores, performing 1358 3rd order Runge-Kutta time steps and
achieved a per-rank throughput of 1.18 
million state updates per second.

In contrast, the adaptive strategy had a final resolution of 41\,883
degrees of freedom per component. The simulation ran for 34.6 
seconds on 8
cores, performing 1178 3rd order Runge-Kutta time steps and achieved a
per-rank throughput of 0.5051 million state updates per second. The mesh
was adapted every tenth simulation cycle. A total of 15.0 
seconds was
spent in the mesh transfer and an additional 8.11 seconds of runtime was
required for re-initializing data structures such as matrices and vectors.
In summary, we report a factor 3.30 speed up in runtime when using our
adaptive mesh refinement strategy.

\subsection{Compressible multi-species Euler Equations}
The multi-species Euler equations are used in modeling a compressible mixture
of two species in mechanical and thermal equilibrium. Let $\bu\eqq(Y_0
\rho, Y_1 \rho ,\bbm, E)^{\mathsf{T}}$ where $\rho$ denotes the mixture
density, $Y_0$ the mass fraction of the first species, $Y_1\eqq 1 - Y_0$
the mass fraction of the second species, $\bbm$ the mixture momentum and
$E$ the total mixture energy. Let $\bv\eqq\bbm / \rho$ be the mixture
velocity. The hyperbolic flux is defined by: $\polf(\bu)\eqq(Y_0\bbm,
Y_1\bbm,\bv\otimes\bbm + p(\bu)\polI_d, \bv(E + p(\bu)))^{\mathsf{T}}$
where $p(\bu)\eqq (\overline{\gamma} - 1)\rho e$ is the ideal mixture
pressure law. The admissible set is defined by $\calA\eqq\{\bu\eqq(Y_0
\rho, Y_1 \rho ,\bbm, E)^{\mathsf{T}} : Y_0\rho > 0,\,Y_1\rho >
0, \,e(\bu)>0, \,\overline{s}(\bu) > s_{\text{min}}\}$ where $e(\bu)\eqq
\rho^{-1}(E  - \frac12 \rho\|\bv\|^2)$ is the mixture specific internal
energy and $\overline{s}$ is the mixture specific entropy.
The respective invariant-domain preserving technique for the multi-species Euler equations with an ideal gas mixture law is outlined in~\citep{IDP_multi_species}.

\subsubsection{Shock-bubble interaction}
We consider a 2D shock-bubble interaction problem for the multi-species Euler equations, which is a common benchmark in the literature (see: \citep{quirk1996dynamics,renac2021entropy,wang2025adaptive} and many others). In this work, we adopt the setup in~\citep[Sec.~6.3.1]{IDP_multi_species}, which consists of a shock wave traveling at Mach 1.43 in air (species $Y_0 \rho$) colliding with a krypton bubble (species $Y_1 \rho$). We refer the reader to~\citep{IDP_multi_species} for the exact setup. A similar setup can be found in~\citep{layes2007quantitative}.

The computational domain is $D = (-0.05,\SI{0.25}{m})\times(0,\SI{0.08}{m})$ and the final time is set to $T = \SI{246}{\mu s}$. 
We show the numerical schlieren for $Y_0 \rho$ with the AMR computational mesh in Figure~\ref{fig:multi-mesh} (top) and with the uniform mesh (below). We note that the AMR mesh at the final time consists of 33,474 cells while the uniform fine mesh consists of 61,440 cells. We observe the IDP projection technique was able to adhere to the respective admissible set and did not produce any negative partial densities for each species.

\begin{figure}
    \centering
    \includegraphics[clip,trim = {0, 0, 0, 0},width=0.85\linewidth]{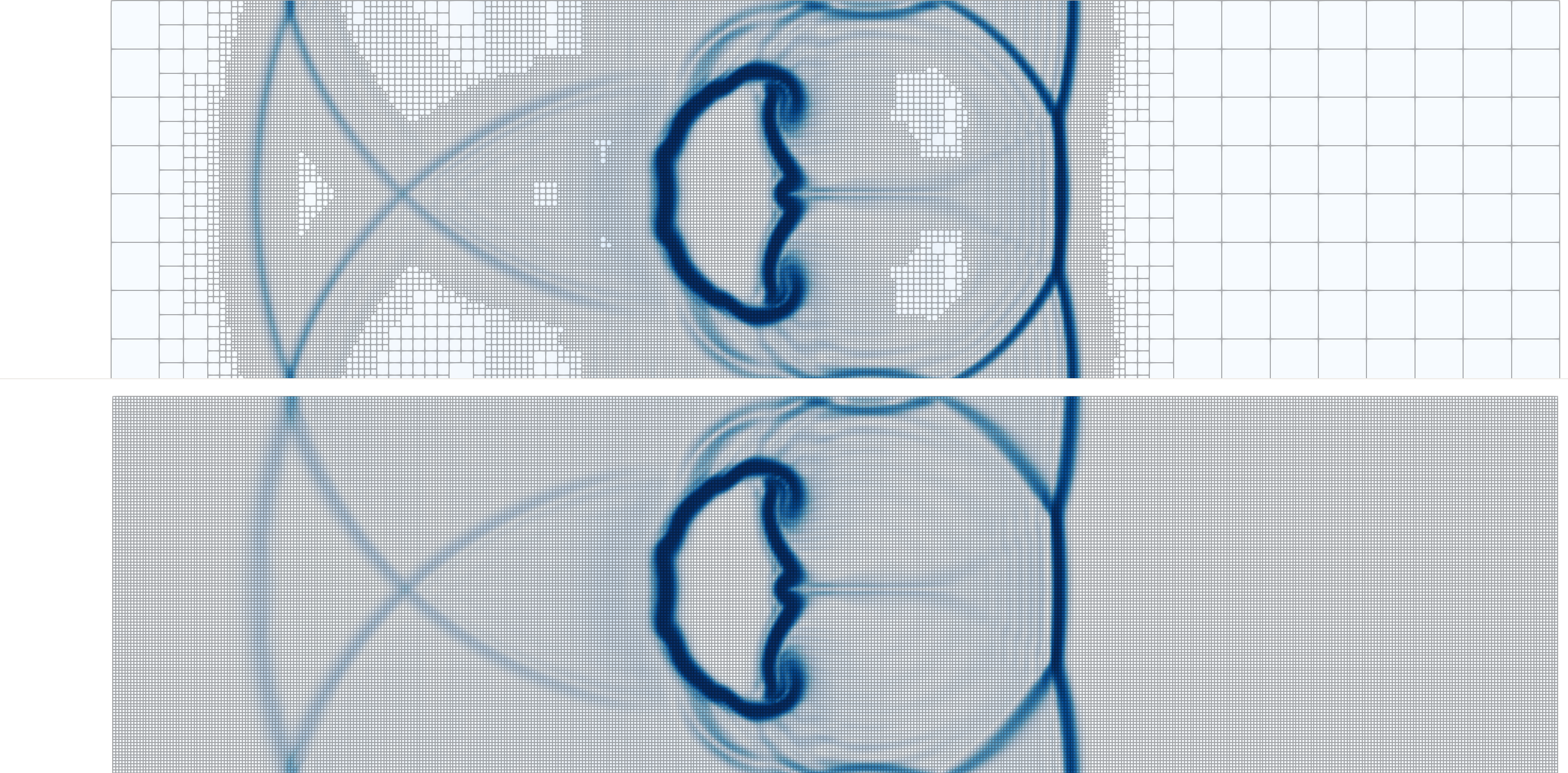}
    \caption{Shock-bubble interaction -- Numerical schlieren for $Y_0\rho$ with AMR mesh (top) and finer uniform mesh (bottom).}
    \label{fig:multi-mesh}
\end{figure}


\section{Conclusion}
\label{sec:conclusion}
We demonstrated a provably invariant-domain preserving, fully adaptive numerical method supported by a robust projection technique. The proposed scheme, by a transformation sequence through compatible solution spaces, guarantees conservation and membership in relevant convex sets. Such properties, essential for the study of hyperbolic systems, ensure the maintenance of physical invariants by the resulting discrete solutions even in the context of adaptive, hierarchical discretizations.
In addition to proving the conservative and IDP properties of our method, we studied a series of benchmark problems over a range of hyperbolic systems, targeting replicable configurations for shallow water and single- and multi-species Euler equations, including an example with a non-ideal (Jones-Wilkins-Lee) equation of state. All benchmark problems confirm the theoretical characteristics in practice and, moreover, the simultaneous ability to allocate computational resources precisely where needed to resolve physics of interest.
Finally, given the mechanical intricacies of the proposed method, we provided detailed high-performance algorithms, with an associated implementation in the open-source code \texttt{ryujin} \cite{maier21b}.

\clearpage
\newpage
\appendix
\section{Algorithmic details}
\label{app:algorithms}
We now provide the algorithms for the methodology described in section~\ref{sec:details}.
\\[2cm]

\begin{algorithm}[h!p]
  \caption{\texttt{register\_data\_attach}\,(cell $K$, cell status $S$, vector $\bU$)}
  \label{alg:register_data_attach}
  \DontPrintSemicolon
  \eIf{$S=$\texttt{cell\_will\_persist} or $S=$\texttt{cell\_will\_be\_refined}}{
    attach $\big\{\bU_i\,:\, i\in\tilde{\mathcal{V}}\text{ s.\,t. }m_i^K\not=0\big\}$ to cell
    $K$\;
    \\
    \Return
  }{
    Let $\{C_k\,:\,k=1,$ \ldots, $2^d\}$ denote the set of children of $K$ that
    will be removed by the coarsening operation. Let
    $\big\{\hat\varphi_i^{K}(\vec x)\big\}$ denote the set of (future) shape
    functions with support on $K$. We set
    $m_i^K\,=\,\int_K\hat\varphi_i^K(\vec x)\text{d}x$, and
    $m_{ij}^K\,=\,\int_K\hat\varphi_i^K(\vec x)\varphi_j^K(\vec
    x)\text{d}x.$
    Recall that $\big\{\tilde\varphi_j^h(\vec
    x)\,:\,j\in\tilde\calV\big\}$ is the Lagrange basis of
    $\tilde{\mathcal{Q}}_{p}(\Pi_h)$.\\[-0.75em]

    \texttt{bounds}$\;\leftarrow\;$initialize\;

    \For{$k=1,$ \ldots, $2^d$}{
      \For{i with $\tilde\varphi_i^h(\vec x)$ having support on $C_k$}{
        \texttt{bounds}$\;\leftarrow\;\text{accumulate}\,(\texttt{bounds},\,\bU_i)$
        \;\\
        \For{j with $\hat\varphi_j^K(\vec x)$ having support on $K$}{
          $\bR_i^K\;\leftarrow\; \bR_i^K \,+\,
          \bU_j\int_K\tilde\varphi_j^h(\bx)\hat\varphi_i^K(\bx)\text{d}x$.
        }
      }
    }
    Compute the quantities $\bU_{i,K}^{\sharp,L}$ and $\bP_{ij}^K$ with
    \eqref{eq:element_wise_ri} and \eqref{eq:element_wise_pi}.\;
    \\
    \For{i with $\hat\varphi_i^K(\vec x)$ having support on $K$}{
      \For{j with $\hat\varphi_j^K(\vec x)$ having support on $K$}{
        $l^K_{ij}\;\leftarrow\;\text{limiter}\,\big(\texttt{bounds}\,;\,
        \bU_{i,K}^{\sharp,L}\,,\,\bP_{ij}^K\big)$
      }
    }
    \For{i with $\hat\varphi_i^K(\vec x)$ having support on $K$}{
      $\bU^{\sharp}_{i,K} \;\leftarrow\; \bU^{\sharp,L}_{i,K}$\;
      \\
      \For{j with $\hat\varphi_j^K(\vec x)$ having support on $K$}{
        $\bU^{\sharp}_{i,K} \;\leftarrow\; \bU^{\sharp}_{i,K} \,+\,
        \kappa\,\min(l^K_{ij},l^K_{ji}) \bP_{ij}^K$
      }
    }
    attach $\big\{\bU^{\sharp}_{i,K}\,:\, i\in\tilde{\mathcal{V}}\text{
    s.\,t. }m_i^K\not=0\big\}$ to cell $K$\;
    \\
    \Return
  }
\end{algorithm}

\begin{algorithm}[p]
  \caption{\texttt{notify\_ready\_to\_unpack}\,(cell $K$, cell status $S$,
    cell data $\big\{\bU^{\sharp}_{i,K}\big\}$)}
  \label{alg:notify_ready_to_unpack}
  \DontPrintSemicolon

  Let $\big\{\tilde\varphi_j^h(\vec x)\,:\,j\in\tilde\calV\big\}$ denote
  the Lagrange basis of $\tilde{\mathcal{Q}}_{p}(\Pi_h)$ constructed for
  the new mesh.\\[-0.75em]

  \eIf{$S=$\texttt{cell\_will\_persist} or
  $S=$\texttt{children\_will\_be\_coarsened}}{
    \For{i with $\tilde\varphi_i^h(\vec x)$ having support on $K$}{
      $\tilde m\bU_i\;\leftarrow \tilde m\bU_i\,+\,m_i^K\bU^{\sharp}_{i,K}$
    }
    \Return
  }{
    Let $\{C_k\,:\,k=1,$ \ldots, $2^d\}$ denote the set of (newly created)
    children of $K$.
    Let $\big\{\hat\varphi_i^{K}(\vec x)\big\}$ denote the set of
    (previous) shape functions with support on $K$.\\[-0.75em]

    \For{$k=1,$ \ldots, $2^d$}{
      \texttt{bounds}$\;\leftarrow\;$initialize\;\\
      \For{i with $\tilde\varphi_i^h(\vec x)$ having support on $C_k$}{
        \texttt{bounds}$\;\leftarrow\;\text{accumulate}\,(\texttt{bounds},\,\bU_i)$
        \;\\
        \For{j with $\hat\varphi_j^K(\vec x)$ having support on $K$}{
          $\bR_i^K\;\leftarrow\; \bR_i^K \,+\,
          \bU_j\int_{C_k}\tilde\varphi_j^h(\bx)\hat\varphi_i^K(\bx)\text{d}x$.
        }
      }
      Compute the quantities $\bU_{i,C_k}^{\sharp,L}$ and $\bP_{ij}^{C_k}$
      with \eqref{eq:element_wise_ri} and \eqref{eq:element_wise_pi}.\;
      \\
      \For{i with $\tilde\varphi_i^{h}(\vec x)$ having support on $C_k$}{
        \For{j with $\tilde\varphi_j^{h}(\vec x)$ having support on $C_k$}{
          $l^{C_k}_{ij}\;\leftarrow\;\text{limiter}\,\big(\texttt{bounds}\,;\,
          \bU_{i,C_k}^{\sharp,L}\,,\,\bP_{ij}^{C_k}\big)$
        }
      }
      \For{i with $\tilde\varphi_i^{h}(\vec x)$ having support on $C_k$}{
        $\bU^{\sharp}_{i,C_k} \;\leftarrow\; \bU^{\sharp,L}_{i,C_k}$\;
        \\
        \For{j with $\tilde\varphi_j^h(\vec x)$ having support on $C_k$}{
          $\bU^{\sharp}_{i,C_k} \;\leftarrow\; \bU^{\sharp}_{i,C_k} \,+\,
          \kappa\,\min(l^{C_k}_{ij},l^{C_k}_{ji}) \bP_{ij}^{C_k}$
        }
        $\tilde m\bU_i\;\leftarrow\; \tilde m\bU_i\,+\,m_i^{C_k}\bU^{\sharp}_{i,C_k}$
      }
    }
    \Return
  }
\end{algorithm}

\begin{algorithm}[p]
  \caption{\texttt{redistribute}\,($\tilde m$, $\tilde m\bU_i$, $\polA\polC_p(\Pi_h)$)}
  \label{alg:redistribute}
  \DontPrintSemicolon

  \For{$i\in\tilde\calV$}{
    $m_i \;\leftarrow\;\tilde m_i$
    \;\\
    $m\bU_i \;\leftarrow\;\tilde m\bU_i$
    \;\\
    $\bU_i \;\leftarrow\;\tilde m\bU_i / \tilde m_i$\;
  }

  \For{$\big(i,\,\{c^i_k\,:\,k\in\tilde{\mathcal{V}}\}\big)\in\polA\polC_p(\Pi_h)$}{
    \texttt{bounds}$_i\;\leftarrow\;$initialize\;\\
    \texttt{bounds}$_i\;\leftarrow\;$accumulate$\,(\texttt{bounds}_i,\,\bU_i)$\;\\
    \For{$k\in\tilde{\mathcal{V}}$ with $c^i_k\not=0$}{
      \vspace{0.2em}
      $\bU_i^{\polA\polC} \;\leftarrow\; \bU_i^{\polA\polC} \,+\,c^i_k\bU_k$\;
      \\
      \texttt{bounds}$_i\;\leftarrow\;$accumulate$\,(\texttt{bounds}_i,\,\bU_k)$\;
    }
    \For{$k\in\tilde{\mathcal{V}}$ with $c^i_k\not=0$}{
      $m\bU_k\;\leftarrow\;m\bU_k\,+\, c_k^i \tilde m\bU_i$
      \;\\
      $m_k\;\leftarrow\;m_k\,+\, c_k^i\tilde m_i$
      \;\\
      $m\bP_k^i\;\leftarrow\; c_k^i\tilde
      m_i\,\big(\bU_k-\bU_i^{\polA\polC}\big)$\;
    }
  }
  \For{$i\not\in\polA\polC_p(\Pi_h)$}{
    $\bU_i \;\leftarrow\;\tilde m\bU_i / \tilde m_i$\;
  }
  \For{$\big(i,\,\{c^i_k\,:\,k\in\tilde{\mathcal{V}}\}\big)\in\polA\polC_p(\Pi_h)$}{
    $l^{i}\;\leftarrow\;1$
    \;\\
    \For{$k\in\tilde{\mathcal{V}}$ with $c^i_k\not=0$}{
      $\kappa_k\;\leftarrow\;\#\{i\,:\,c^i_k\not=0\}$\;\\
      $l^{i}_{k}\;\leftarrow\;\text{limiter}\,\big(\texttt{bounds}_i\,;\,
      m_k^{-1}m\bU_{k}\,,\,m_k^{-1}\kappa_k\,m\bP_{k}^{i}\big)$
      \;\\
      $l^{i}\;\leftarrow\;\min(l^{i},l^{i}_k)$\;
    }
    \For{$k\in\tilde{\mathcal{V}}$ with $c^i_k\not=0$}{
      $m\bU_k\;\leftarrow\;m\bU_k\,+\, l^i m\bP^i_k$
    }
  }
  \For{$i\not\in\polA\polC_p(\Pi_h)$}{
    $\bU_i \;\leftarrow\;\tilde m\bU_i / \tilde m_i$\;
  }
  \For{$\big(i,\,\{c^i_k\,:\,k\in\tilde{\mathcal{V}}\}\big)\in\polA\polC_p(\Pi_h)$}{
    $m_i\;\leftarrow\;0$\;\\
    $\bU_i \;\leftarrow\;0$\;\\
    \For{$k\in\tilde{\mathcal{V}}$ with $c^i_k\not=0$}{
      $\bU_i \;\leftarrow\;\bU_i\,+\,c^i_k\bU_k$\;
    }
  }
  \Return $(m,\bU)$
\end{algorithm}

\clearpage
\newpage

\bibliographystyle{abbrvnat}

\end{document}